\PassOptionsToPackage{dvipsnames,svgnames,x11names}{xcolor}

\documentclass[12pt,reqno]{amsart}

\usepackage{amsmath}
\usepackage{amssymb}
\usepackage{amsthm}
\usepackage{bbm}
\usepackage{mathtools}
\usepackage{nicefrac}

\usepackage[main=english,brazilian]{babel}
\usepackage[T1]{fontenc}
\usepackage[utf8]{inputenc} 

\usepackage{a4wide}
\usepackage{microtype}

\usepackage[all]{xy}
\usepackage{tikz-cd}

\usepackage{aligned-overset}
\usepackage{enumitem}
\usepackage{multicol}
\usepackage{scalefnt}
\usepackage{xspace}

\usepackage{fixme}
\usepackage{url}

\usepackage{xcolor}
\usepackage{hyperref}
\usepackage{cleveref} 

\newtheorem{theorem}{Theorem}[section]
\newtheorem{proposition}[theorem]{Proposition}
\newtheorem{lemma}[theorem]{Lemma}
\newtheorem{corollary}[theorem]{Corollary}
\theoremstyle{definition}
\newtheorem{definition}[theorem]{Definition}
\newtheorem{remark}[theorem]{Remark}

\newcommand*{\C}{\mathbb C}
\newcommand*{\Z}{\mathbb Z}
\newcommand*{\B}{\mathcal B}
\newcommand*{\Hh}{\mathcal H}
\newcommand*{\cL}{\mathcal L}
\newcommand*{\Star}{$^*$\nobreakdash}
\newcommand*{\id}{\operatorname{id}}
\newcommand*{\Pic}{\operatorname{Pic}}
\newcommand*{\dom}{\operatorname{dom}}

\title{Spectral geometry of saturated Fell bundles}
\author{Stefan Wagner}
\address{Blekinge Institute of Technology, Department of Mathematics,
Karlskrona, Sweden}
\email{stefan.wagner@bth.se}

\begin{document}

\begin{abstract}
Let $\B=(B_g)_{g\in G}$ be a saturated Fell bundle over a countable discrete group, equipped with a spectral triple.
We give explicit sufficient conditions for lifting this triple to the reduced cross-sectional C\Star-algebra $C_r^*(\B)$.
Using finite projective realizations of the homogeneous Morita equivalence $B_e$-bimodules, we transport the base Dirac operator to the individual fibres and assemble the resulting operators into a fibrewise Dirac operator.
We introduce a spectral-regularity condition ensuring bounded commutators and give a sufficient criterion in terms of uniform derivative bounds on smooth factor-system data.
Combining the fibrewise operator with a length operator associated with a proper translation-bounded matrix-valued length function yields the lifted spectral triple.
When $B_e$ is separable, the lifted triple, with a suitable grading, represents the Kasparov product of the vertical class with the class of the unit-fibre triple.
We also establish a summability criterion and sufficient conditions for stability under central cocycle twisting.
A range of examples illustrates the scope of the construction, including twisted crossed products, free compact Abelian actions, and Fell bundles with nonfree homogeneous fibres.
\end{abstract}

\maketitle

\section{Introduction}

Fell bundles provide a flexible framework for decomposing C\Star-algebras into homogeneous fibres. 
Over discrete groups, their cross-sectional algebras form a broad class of graded C\Star-algebras encompassing group C\Star-algebras, crossed products, twisted crossed products, and many naturally graded examples; see~\cite{Ex17,Fell1,Fell2}. 
In the saturated case, the homogeneous fibres interact through multiplication in the strongest possible sense: for all $g,h\in G$, the linear span of
$B_gB_h$ is dense in $B_{gh}$. 
In particular, each fibre $B_g$ is a Morita equivalence $B_e$-bimodule. 
This raises the natural question of whether geometric data carried by the unit fibre can be transported along the homogeneous fibres and assembled into geometric data on the cross-sectional algebra.

Spectral triples provide a natural analytic framework for encoding noncommutative geometric data. 
Modelled on the Dirac operator of a compact Riemannian spin manifold, they capture first-order differential structure through commutators with the algebra, while spectral properties of the Dirac operator, such as compact resolvent and summability, reflect global geometric features. 
They therefore provide a natural setting in which to ask whether geometric data on the unit fibre can be transported to the cross-sectional algebra.

Spectral-geometric constructions have been developed for a wide range of operator-algebraic structures; see, for example,~\cite{AGI17,AnGuIsRu22,BeMaRe10,Cac23,CacMes19,GaGr13,GoffengMesland2018,Ka25,KaSui20,SchWa20,SiVe15} and the references therein.
Of particular relevance here are spectral triples on group C\Star-algebras, crossed products and twisted crossed products, and noncommutative principal bundles arising from free compact group actions; see, for example,~\cite{AnGuIsRu22,Cac23,CacMes19,SchWa17} and the references therein. 
Related constructions have also been studied for graph C\Star-algebras and other naturally graded C\Star-algebras; see, for example,~\cite{GoffengMesland2018,PaskRennie2006} and the references therein.

A recurring principle in these constructions is that spectral data on a distinguished~subalgebra are combined with geometric information associated with a group action or grading. 
For crossed products by discrete groups, this additional information can be encoded by a length function on the group, while for compact group actions it is supplied by the geometry of the acting group.

For Fell bundles, Masters constructed unbounded Kasparov modules for fissured Fell bundles over locally compact groups using proper translation-bounded matrix-valued weights; see~\cite{MastersThesis}.
In the saturated case considered here, her construction provides the vertical length operator used below.
Moreover, in~\cite[Sec.~II.4, especially Thm.~II.4.6]{MastersThesis}, she studies Fell bundles over group extensions and, under suitable compatibility assumptions, combines weights on the normal subgroup and the quotient to obtain an unbounded representative of the Kasparov product through the corresponding partial cross-sectional algebra.

The present article considers a related lifting problem, starting from a given spectral triple on the unit fibre of a saturated Fell bundle over a countable discrete group.
Its focus is the explicit transport of the base operator through finite projective realizations of the homogeneous fibres and the formulation of factor-system conditions ensuring bounded horizontal commutators.
These conditions allow the resulting fibrewise operator to be combined with the vertical length operator to obtain a spectral triple on the reduced cross-sectional algebra.

More precisely, let $\B=(B_g)_{g \in G}$ be a saturated Fell bundle over a countable discrete group, with unital unit fibre $B := B_e$, and let $\mathbb D_B=(B_0,\pi_B,\Hh_B,D_B)$ be a spectral triple on $B$. 
The construction proceeds as follows.

The starting point is an isometric characterization of saturation (Lemma~\ref{lem:isometric-saturation}).
More precisely, $\B$ is saturated if and only if, for each $g \in G$, there exist a finite-dimensional Hilbert space $H_g$ and an isometry
\[
    s(g) \in B_{g^{-1}}\otimes H_g,
    \qquad
    s(g)^* s(g) = 1_B.
\]
Fix such a family, with $H_e = \C$ and $s(e) = 1_B$. 
For each $g \in G$, the projection
\[
    p(g) := s(g)s(g)^* \in B \otimes \cL(H_g)
\]
gives a realization of the homogeneous fibre $B_g$ as the finitely generated projective $B$-module
\[
    p(g)(B \otimes H_g).
\]
These projective realizations (Corollary~\ref{cor:projective-fibres}) provide the coordinates in which both the Fell-bundle structure and the lifted spectral data are described.

The chosen family determines factor-system data as follows. 
For each $g \in G$, one obtains a \Star-homomorphism
\[
    \gamma_g: B \rightarrow B \otimes \cL(H_g),
    \qquad
    \gamma_g(x) := s(g) x s(g)^*,
\]
while, for all $g,h \in G$, one obtains an operator
\[
    \omega(g,h)
    :=
    s(gh) s(g)^* s(h)^*
    \in
    B \otimes \cL(H_ g\otimes H_h,H_{gh}).
\]
The maps $\gamma_g$ encode the $B$-module structure of the individual fibres, while the operators $\omega(g,h)$ encode their multiplication.
If these factor-system data are smooth relative to~$B_0$, they determine dense homogeneous subspaces $B_g^0\subseteq B_g$, $g\in G$, and hence the dense \Star-subalgebra
\[
    C_c(\B)_0 := \bigoplus_{g\in G}^{\mathrm{alg}} B_g^0 \subseteq C_r^*(\B).
\]

Under the same smoothness assumption, the Dirac operator $D_B$ can be lifted to each homogeneous fibre. 
For each $g \in G$, let
\[
    W_g: p(g)(\Hh_B \otimes H_g) \rightarrow B_g \otimes_B\Hh_B
\]
be the unitary induced by the projective realization. 
Since $p(g) \in B_0 \otimes \cL(H_g)$, the compressed operator
\[
    p(g)(D_B \otimes \id)p(g)
\]
is self-adjoint with compact resolvent, and its transport through $W_g$ defines
\[
    D_g := W_g p(g)(D_B \otimes \id)p(g) W_g^*.
\]
The direct sum of these operators defines the fibrewise Dirac operator
\[
    D_{\mathrm{fib}} := \bigoplus_{g \in G} D_g
\]
on 
\[
    \Hh_{\B} := \ell^2(\B) \otimes_B \Hh_B.
\]

The principal analytic issue at this stage is the boundedness of commutators with $D_{\mathrm{fib}}$. 
For this purpose, a \emph{spectral-regularity} condition is introduced, expressing domain
invariance together with uniform boundedness of the fibrewise commutators.
A sufficient criterion can be formulated directly in terms of the factor-system data. 
For all $h \in G$ and $x \in B_0$, put
\[
    \partial\gamma_h(x)
    :=
    [D_B \otimes \id,\gamma_h(x)],
\]
and, for all $g,h \in G$, put
\[
    \partial\omega(g,h)
    :=
    (D_B \otimes \id) \omega(g,h)
    -
    \omega(g,h) (D_B \otimes \id).
\]
The factor-system criterion (Theorem~\ref{thm:factor-criterion}) shows that spectral regularity follows if
\[
    \sup_{h \in G} \| \partial\gamma_h(x) \| < \infty
\]
for all $x \in B_0$, and
\[
    \sup_{h\in G} \| \partial\omega(g,h) \| < \infty
\]
for all $g \in G$. 
Notably, these conditions are closely related to the factor-system conditions appearing in~\cite{SchWa20} in the setting of noncommutative principal bundles.

The contribution from the grading group is encoded by a proper
translation-bounded matrix-valued length function $\ell: G \rightarrow \cL(V)_{\mathrm{sa}}$, where $V$ is finite-dimensional. 
It determines a diagonal length operator $M_\ell$ on $\Hh_{\B} \otimes V$. 
The fibrewise Dirac operator and the length operator are then combined to form
\[
    D_{\B,\ell}
    :=
    D_{\mathrm{fib}} \otimes \id_V \otimes \sigma_1
    +
    M_\ell \otimes \sigma_2
\]
on $\Hh_{\B} \otimes V \otimes \C^2$, where $\sigma_1$ and $\sigma_2$ are the first two Pauli matrices. 
Properness of $\ell$~provides the decay needed for compact resolvent, while translation boundedness controls the commutators arising from the grading direction. 
Together with spectral regularity, this gives a spectral triple
(Theorem~\ref{thm:main})
\[
    \mathbb D_{\B,\ell}
    =
    \left(
        C_c(\B)_0,
        \pi_{\B} \otimes \id_V \otimes \id_{\C^2},
        \Hh_{\B} \otimes V \otimes \C^2,
        D_{\B,\ell}
    \right)
\]
on $C_r^*(\B)$.

The construction may therefore be viewed as combining two complementary directions. 
The matrix-valued length function gives rise to the vertical geometry studied by Masters~\cite{MastersThesis}, whereas the fibrewise operator constructed here transports the spectral geometry of the unit fibre through the Morita equivalence fibres. 

When $B$ is separable, the lifted spectral triple, equipped with a suitable grading, represents the Kasparov product of the vertical class defined by $M_\ell$ with the class of the spectral triple on the unit fibre
(Proposition~\ref{prop:kasparov-product}).

A summability criterion is obtained in terms of uniform summability bounds for the fibre operators $D_g$ and the growth of the matrix-valued length function (Proposition~\ref{prop:summability}).
The behaviour of the construction under central cocycle twisting is also considered.
Sufficient conditions are given under which spectral regularity is preserved by the twist, and hence the same construction produces a spectral triple on the reduced cross-sectional algebra of the cocycle-twisted Fell bundle
(Theorem~\ref{thm:cocycle-twist}).

For twisted crossed products, the resulting spectral triples recover the construction of~\cite{AnGuIsRu22}.
A further important class arises from free actions of compact Abelian groups.
If~$K$ is compact Abelian and acts freely on a C\Star-algebra, then its spectral subspaces form a saturated Fell bundle over the discrete dual group $\widehat K$.
In this setting, the fibrewise lifting procedure is closely related to the construction developed in~\cite{SchWa20}.
The Fell-bundle formulation recasts the relevant spectral data in terms of the homogeneous fibres and places both constructions in a common framework applicable to arbitrary countable discrete grading groups.

A concrete application is obtained from the Hopf line bundle over $S^2$, pulled back to $S^2\times S^1$ (Section~\ref{sec:free-group-example}).
Its tensor powers, indexed through a homomorphism $F_2\to\Z$, form a saturated Fell bundle over the non-Abelian free group $F_2$ with infinitely many pairwise nonisomorphic
homogeneous modules.
Compatible tensor-product connections give the uniform horizontal commutator estimates directly through the Leibniz rule.
A signed word-length function then yields a lifted spectral triple.

This example also clarifies the relation between the analytic hypotheses used in the paper.
Spectral regularity holds, while the uniform derivative bounds in the factor-system criterion fail: such bounds would uniformly bound the curvatures of the rank-one Grassmann connections,
contradicting the unbounded Chern numbers of the tensor powers.
Thus the factor-system criterion is sufficient but not necessary.

Further examples include Fell line bundles, finite Fell bundles, free compact Abelian actions, quantum tori, the quantum Hopf fibration, and the gauge grading of the Cuntz algebras.

\subsection*{Organization of the paper}

Section~\ref{sec:prelim} recalls the necessary background on Fell bundles, reduced cross-sectional C\Star-algebras, spectral triples, and Dirac operators on finitely generated projective modules.

Section~\ref{sec:isometric-factor-systems} develops the structure of saturated Fell bundles used throughout the paper. 
Saturation is characterized in terms of finite-dimensional isometries, the homogeneous fibres are realized as finitely generated projective modules, and these realizations are organized into factor-system data.

Section~\ref{sec:lifts} develops the fibrewise lifting of the spectral data on the unit fibre. 
The smooth algebra, the lifted representation, and the fibrewise Dirac operator are constructed, and spectral regularity is introduced together with a criterion formulated directly in terms of the factor-system data.

In Section~\ref{sec:construction}, a matrix-valued length function on the grading group is used to construct the length operator, which is combined with the fibrewise Dirac operator to obtain the lifted Dirac operator.
The main Fell-bundle lifting theorem is established, followed by a Kasparov-product interpretation and a summability criterion.

Section~\ref{sec:cocycle-twisting} studies central cocycle twists and gives conditions under which the lifted spectral triple is preserved.

Finally, Section~\ref{sec:examples} illustrates the construction through trivial Fell bundles, twisted crossed products, Fell line bundles, finite Fell bundles, and a free-group
example with nonfree fibres.
It also discusses free compact Abelian actions, quantum tori, the quantum Hopf fibration, and the gauge grading of the Cuntz algebras.

\section{Preliminaries}\label{sec:prelim}

Throughout, $G$ denotes a countable discrete group with identity element $e$.

For a subset $X$ of a Banach space, we denote its closed linear span by $[X]$.

All Hilbert modules are right Hilbert modules.  
For a Hilbert $B$-module $E$, we write $\cL_B(E)$ and $\mathcal K_B(E)$ for the adjointable and compact operators, respectively.

\subsection{Fell bundles over discrete groups}

\begin{definition}[{\cite[Chap.~VIII, Sec.~16.2]{Fell2}}]
A \emph{Fell bundle} over $G$ is a family $\B=(B_g)_{g \in G}$ of Banach spaces, called the \emph{fibres} of $\B$, equipped with
\begin{itemize}
    \item bilinear maps $m_{g,h}: B_g \times B_h \to B_{gh}$, $(x,y) \mapsto xy$ for all $g,h \in G$, and 
    \item involution maps ${}^*: B_g \to B_{g^{-1}}$, $x \mapsto x^*$ for all $g \in G$,
\end{itemize}
satisfying the usual C\Star-algebraic bundle axioms.

A Fell bundle $\B=(B_g)_{g\in G}$ is called \emph{saturated} if $[B_gB_h] = B_{gh}$ for all $g,h \in G$.
\end{definition}

Let $\B = (B_g)_{g\in G}$ be a Fell bundle and put $B := B_e$.
Then $B$ is a C\Star-algebra, called~the \emph{unit fibre}.
For each $g \in G$, the fibre $B_g$ is a Hilbert $B$-bimodule with respect to the
canonical left and right $B$-valued inner products
\[
    {}_B\langle x,y\rangle := xy^*
    \qquad
    \text{and}
    \qquad
    \langle x,y\rangle_B := x^*y.
\]
For all $g,h \in G$, the bilinear map $m_{g,h}$ induces an isometric $B$-bimodule map
\[
    m_{g,h}: B_g\otimes_B B_h \rightarrow B_{gh},
    \qquad
    x \otimes y \mapsto xy,
\]
whose range is $[B_gB_h]$.
If $\B$ is saturated, then $m_{g,h}$ is unitary for all $g,h \in G$.
Consequently, each fibre $B_g$ is a Morita equivalence $B$-bimodule.

Let $C_c(\B)$ denote the space of finitely supported sections of $\B$, which may be identified with the algebraic direct sum 
\[
    C_c(\B) = \bigoplus_{g \in G}^{\mathrm{alg}} B_g.
\]
For $\xi \in C_c(\B)$ and $g \in G$, let $\xi_g \in B_g$ be its $g$-component. 
Then $C_c(\B)$ becomes a \Star-algebra with multiplication and involution defined, for $\xi,\eta \in C_c(\B)$, by
\[
    (\xi\eta)_g := \sum_{h \in G} \xi_h \eta_{h^{-1}g}
    \qquad
    \text{and}
    \qquad
    (\xi^*)_g := \xi_{g^{-1}}^*
\]
for all $g \in G$.
Equip $C_c(\B)$ with the $B$-valued inner product
\[
    \langle \xi,\eta\rangle_B := \sum_{g \in G} \xi_g^* \eta_g.
\]
Its completion is the Hilbert $B$-module
\begin{equation}\label{eq:module-decomposition}
    \ell^2(\B) := \bigoplus_{g \in G} B_g.
\end{equation}
Left convolution induces the regular representation
\[
    \Lambda_{\B}:
    C_c(\B) \rightarrow \cL_B(\ell^2(\B)).
    \qquad
    (\Lambda_{\B}(\xi)\eta)_g := \sum_{h \in G} \xi_h \eta_{h^{-1}g}.
\]
The \emph{reduced cross-sectional C\Star-algebra} of $\B$ is defined by
\[
    C_r^*(\B) := \overline{\Lambda_{\B}(C_c(\B))}^{\|\cdot\|}
    \subseteq
    \cL_B(\ell^2(\B)).
\]

\subsection{Spectral triples}

Let $C$ be a unital C\Star-algebra.
A~\emph{spectral triple} on $C$ is a quadruple $\mathbb D_C=(C_0,\pi_C,\Hh_C,D_C)$, where $C_0 \subseteq C$ is a dense unital \Star-subalgebra, $\pi_C: C \to \cL(\Hh_C)$ is a faithful unital \Star-representation on a Hilbert space $\Hh_C$, and $D_C$ is a possibly unbounded self-adjoint operator on $\Hh_C$ such that
\begin{enumerate}
    \item $D_C$ has compact resolvent,
    \item $\pi_C(C_0) \dom(D_C) \subseteq \dom(D_C)$,
    \item $[D_C,\pi_C(c)] \in \cL(\Hh_C)$ for all $c \in C_0$.
\end{enumerate}



Let $E$ be a finitely generated projective right Hilbert $C$-module and suppose that $E \cong pC^n$ for some projection $p \in M_n(C_0)$. 
Via the induced representation of $M_n(C)$ on $\Hh_C^n$, the commutator $[D_C \otimes \id,p]$ is bounded. 
It follows from~\cite[Lem.~2.1]{SchWa20} that
\[
    D_E := p (D_C \otimes \id) p
\]
is an unbounded self-adjoint operator on $p \Hh_C^n$ with domain $p \dom(D_C)^n$ and compact~resolvent.


\section{Saturated Fell bundles and factor systems}\label{sec:isometric-factor-systems}

A structural description of saturated Fell bundles used in the subsequent
construction is developed in this section.
The homogeneous fibres are first realized as finitely generated projective modules, and these realizations are then organized into factor-system data
encoding the Fell-bundle structure.

Let $\B=(B_g)_{g\in G}$ be a saturated Fell bundle with unital unit fibre
$B := B_e$.

\subsection{Saturation and projective fibres}

For each $g \in G$, there exist $n_g \in \mathbb N$ and elements
$x_1,\ldots,x_{n_g} \in B_g$ such that
\[
    \sum_{i=1}^{n_g} x_i x_i^* = 1_B;
\]
see, for example,~\cite[Lem.~A.1]{SchWa17}.

The existence of such finite families can be expressed equivalently in terms of isometries, which will provide the coordinates used throughout the construction:

\begin{lemma}[Saturation via isometries]\label{lem:isometric-saturation}
Suppose that $B$ is unital. The following assertions are equivalent:
\begin{enumerate}[label=(\alph*)]
    \item The Fell bundle $\B$ is saturated.
    \item For each $g\in G$, there exist a finite-dimensional Hilbert space $H_g$ and an isometry
    \[
        s(g) \in B_{g^{-1}} \otimes H_g.
    \]
\end{enumerate}
\end{lemma}
\begin{proof}
Suppose first that \textup{(a)} holds and fix $g \in G$. 
By the preceding observation, there exist $x_1,\ldots,x_n \in B_g$ such that
$\sum_{j=1}^n x_j x_j^* = 1_B$. 
Hence
\[
    s(g) := (x_1^*,\ldots,x_n^*)^{\mathsf T} \in B_{g^{-1}} \otimes \C^n
\]
satisfies $s(g)^* s(g) = \sum_{j=1}^n x_j x_j^* = 1_B$.
Thus \textup{(b)} holds.

Conversely, suppose that \textup{(b)} holds. 
Let $g,h \in G$ and $z \in B_{gh}$. 
Apply \textup{(b)} to $h^{-1}$ and let $n := \dim H_{h^{-1}}$. 
After choosing an orthonormal basis of $H_{h^{-1}}$, write
\[
    s(h^{-1}) = (y_1,\ldots,y_n)^{\mathsf T},
\]
where $y_j \in B_h$ for all $j=1,\ldots,n$. 
Since $s(h^{-1})$ is an isometry, $\sum_{j=1}^n y_j^* y_j = 1_B$. 
Hence
\[
    z = \sum_{j=1}^n (zy_j^*) y_j \in [B_gB_h].
\]
Thus $B_{gh} \subseteq [B_gB_h]$. 
The Fell-bundle axioms give the reverse inclusion, so $[B_gB_h] = B_{gh}$.
From this it follows that $\B$ is saturated.
\end{proof}

\begin{corollary}\label{cor:projective-fibres}
Suppose that $B$ is unital and that $\B$ is saturated. 
For each $g \in G$, choose an isometry $s(g) \in B_{g^{-1}} \otimes H_g$ as in Lemma~\ref{lem:isometric-saturation}, and put
\[
    p(g) := s(g) s(g)^* \in B \otimes \cL(H_g).
\]
Then, for each $g \in G$, $p(g)$ is a projection and
\begin{equation}\label{eq:fibre-decomposition}
    \Phi_g: p(g) ( B\otimes H_g) \rightarrow B_g,
    \qquad
    \Phi_g(\xi) := s(g)^* \xi,
\end{equation}
is a unitary isomorphism of Hilbert $B$-modules. 
In particular, $B_g$ is finitely generated and projective for each $g \in G$.
\end{corollary}
\begin{proof}
Fix $g\in G$. 
The isometry property of $s(g)$ gives $p(g)^* = p(g)$ and $p(g)^2 = p(g)$, so $p(g)$ is a projection.

Since $p(g)s(g)x = s(g)x$ and $\Phi_g(s(g)x) = x$ for all $x \in B_g$, the map $\Phi_g$ is surjective. 
Moreover, for all $\xi,\eta \in p(g)(B\otimes H_g)$,
\[
    \langle\Phi_g(\xi),\Phi_g(\eta)\rangle_B
    =
    \xi^* s(g) s(g)^* \eta
    =
    \xi^* p(g) \eta
    =
    \langle\xi,\eta\rangle_B.
\]
Thus $\Phi_g$ is unitary.
\end{proof}

\subsection{Factor-system data}

The projective realizations of the fibres obtained above can be organized
into factor-system data encoding both the $B$-module structure of the
individual fibres and the multiplication of the Fell bundle.
Fix, for each $g \in G$, a finite-dimensional Hilbert space $H_g$ and an
isometry
\[
    s(g) \in B_{g^{-1}} \otimes H_g
\]
as in Lemma~\ref{lem:isometric-saturation}, with $H_e = \C$ and $s(e) = 1_B$.
These choices give rise to \emph{factor-system data over $B$} describing
the fibres of $\B$ together with their multiplication and involution:

For each $g \in G$, define the \Star-homomorphism
\[
    \gamma_g: B \rightarrow B \otimes \cL(H_g),
    \qquad
    \gamma_g(x) := s(g) x s(g)^*.
\]
Note that $\gamma_g(1_B) = s(g) s(g)^* = p(g)$.

For all $g,h \in G$, define the operator
\[
    \omega(g,h) := s(gh) s(g)^* s(h)^* \in B \otimes \cL(H_g \otimes H_h,H_{gh}),
\]
where the isometries are amplified in the natural way. 

For each $g \in G$ and finite-dimensional Hilbert spaces $K$ and $L$, the map $\gamma_g$ is extended entrywise to
\[
    \gamma_g: B \otimes \cL(K,L) \rightarrow B \otimes \cL(K \otimes H_g,L \otimes H_g).
\]
The operators $\omega(g,h)$ are likewise amplified by tensoring with identity operators on auxiliary Hilbert-space factors whenever necessary. 
With these conventions,
\[
    \omega(g,h) \omega(g,h)^* = p(gh)
    \qquad
    \text{and}
    \qquad
    \omega(g,h)^* \omega(g,h) = \gamma_h(p(g))
\]
for all $g,h \in G$. 
Moreover,
\begin{equation}\label{eq:factor-compatibility}
    \gamma_{gh}(x) \omega(g,h)
    =
    \omega(g,h) \gamma_h(\gamma_g(x))
\end{equation}
for all $g,h \in G$ and $x \in B$, and
\begin{equation}\label{eq:factor-cocycle}
    \omega(g,hk) (1_{H_g} \otimes \omega(h,k))
    =
    \omega(gh,k) \gamma_k(\omega(g,h))
\end{equation}
for all $g,h,k \in G$. 
The identities~\eqref{eq:factor-compatibility} and~\eqref{eq:factor-cocycle} are referred to as the \emph{coaction identity} and the \emph{cocycle identity}, respectively.

The maps $\gamma_g$, $g \in G$, and the operators $\omega(g,h)$, $g,h \in G$, determine the algebraic operations on the fibres in the following coordinates:

For each $g \in G$, define
\[
    a_g: B \otimes H_g \rightarrow B_g,
    \qquad
    a_g(x \otimes \eta) := s(g)^* (x \otimes \eta).
\]
The map $a_g$ is surjective, since $a_g(s(g) x) = x$ for all $x \in B_g$. 
For all $g,h \in G$, $x,y \in B$, $\eta \in H_g$, and $\vartheta \in H_h$, one has
\begin{equation}\label{eq:factor-multiplication}
    a_g(x \otimes \eta) a_h(y \otimes \vartheta) = a_{gh}(\omega(g,h) \gamma_h(x)_{13} (y \otimes \eta \otimes \vartheta)),
\end{equation}
where $\gamma_h(x)_{13}$ acts on the $B$- and $H_h$-components and trivially on $H_g$.

For each $g \in G$, define the antilinear map
\[
    J_g: B \otimes H_g \rightarrow B \otimes H_{g^{-1}},
    \qquad
    J_g(\xi) := s(g^{-1}) a_g(\xi)^*.
\]
Then $a_g(\xi)^* = a_{g^{-1}}(J_g(\xi))$ for all $g \in G$ and $\xi \in B \otimes H_g$.

For $g \in G$ and $\eta \in H_g$, let $\iota_\eta: H_{g^{-1}} \rightarrow H_g \otimes H_{g^{-1}}$, $\vartheta \mapsto \eta \otimes \vartheta$.
Then
\begin{equation}\label{eq:J-factor-system}
    J_g(x \otimes \eta) = \gamma_{g^{-1}}(x^*) (\id_B \otimes \iota_\eta^*) \omega(g,g^{-1})^*
\end{equation}
for all $g \in G$, $x \in B$, and $\eta \in H_g$.

\section{Fibrewise lifting of spectral data}\label{sec:lifts}

Let $\B=(B_g)_{g\in G}$ be a saturated Fell bundle with unital unit fibre $B:=B_e$, and let $\mathbb D_B=(B_0,\pi_B,\Hh_B,D_B)$ be a spectral triple on $B$.
To construct a lift of the spectral data from $B$ to $C_r^*(\B)$, the following compatibility condition is imposed:

\begin{definition}\label{def:smooth-factor-system}
Factor-system data for $\B$ as in Section~\ref{sec:isometric-factor-systems} are called \emph{smooth} relative to $B_0$ if
\begin{enumerate}
    \item $\gamma_g(B_0) \subseteq B_0\otimes\cL(H_g)$ for all $g\in G$;
    \item $\omega(g,h) \in B_0 \otimes \cL(H_g \otimes H_h,H_{gh})$ for all $g,h \in G$.
\end{enumerate}
\end{definition}

Fix factor-system data for $\B$ that are smooth relative to $B_0$.
Then $p(g) \in B_0 \otimes \cL(H_g)$ for all $g \in G$.
Moreover, by~\eqref{eq:J-factor-system}, 
\begin{equation}\label{eq:smooth-J}
    J_g(B_0 \otimes H_g) \subseteq B_0 \otimes H_{g^{-1}}
\end{equation}
for all $g \in G$.

\subsection{The lifted algebra}

The smooth factor-system data determine a natural dense \Star-subalgebra of $C_r^*(\B)$.
Indeed, for each $g\in G$, put
\[
    B_g^0 := a_g(B_0 \otimes H_g) \subseteq B_g.
\]
Since $B_0$ is dense in $B$ and $a_g$ is surjective, $B_g^0$ is dense in $B_g$. 
By~\eqref{eq:factor-multiplication}, $B_g^0 B_h^0 \subseteq B_{gh}^0$ for all $g,h \in G$.
Furthermore,~\eqref{eq:smooth-J} gives $(B_g^0)^* \subseteq B_{g^{-1}}^0$ for all $g \in G$.
Applying the same inclusion to $g^{-1}$ yields the reverse inclusion, and hence $ (B_g^0)^* = B_{g^{-1}}^0$ for all $g \in G$.
It follows that
\[
    C_c(\B)_0 := \bigoplus_{g\in G}^{\mathrm{alg}} B_g^0
\]
is a dense \Star-subalgebra of $C_r^*(\B)$.

\subsection{The lifted representation}\label{sec:lifted_representation}

The representation $\pi_B$ turns $\Hh_B$ into a left $B$-module and thereby allows the regular representation of $C_r^*(\B)$ to be lifted to a Hilbert space. 
Indeed, the interior tensor product $\Hh_{\B} := \ell^2(\B) \otimes_B \Hh_B$ is a Hilbert space, with inner~product determined by
\[
    \langle \xi \otimes u,\eta \otimes v \rangle
    :=
    \langle u,\pi_B(\langle \xi,\eta \rangle_B)v \rangle_{\Hh_B}
\]
for all $\xi,\eta \in \ell^2(\B)$ and $u,v \in \Hh_B$.

The regular representation of $C_r^*(\B)$ on $\ell^2(\B)$ induces a faithful \Star-representation
\[
    \pi_{\B}: C_r^*(\B) \rightarrow \cL(\Hh_{\B}), 
    \qquad 
    \pi_{\B}(x) := \Lambda_{\B}(x) \otimes \id_{\Hh_B}. 
\]
Indeed, suppose that $T \in \cL_B(\ell^2(\B))$ satisfies $T \otimes \id_{\Hh_B} = 0$. 
Then
\[
    \langle u,\pi_B(\langle T\xi,T\xi \rangle_B)u \rangle_{\Hh_B} = 0
\]
for all $\xi \in \ell^2(\B)$ and $u \in \Hh_B$, and hence $\pi_B(\langle T\xi,T\xi \rangle_B) = 0$ for all $\xi \in \ell^2(\B)$. 
Since $\pi_B$ is faithful, $\langle T\xi,T\xi \rangle_B = 0$ for all $\xi \in \ell^2(\B)$, which in turn yields $T=0$. 
Thus the map
\[
    \rho: \cL_B(\ell^2(\B)) \rightarrow \cL(\Hh_{\B}),
    \qquad
    \rho(T) := T \otimes \id_{\Hh_B},
\]
is faithful. 
Together with the faithfulness of $\Lambda_{\B}$ and the identity
$\pi_{\B} = \rho \circ \Lambda_{\B}$, this shows~that $\pi_{\B}$ is faithful.

By~\eqref{eq:module-decomposition}, the Hilbert space $\Hh_{\B}$ admits the decomposition 
\begin{equation}\label{eq:Hilbert-space-decomposition} 
    \Hh_{\B} \cong \bigoplus_{g\in G} (B_g \otimes_B \Hh_B). 
\end{equation} 
With respect to this decomposition, an element $x \in B_g$, with $g \in G$, viewed as the section supported at $g$, acts by 
\[
    \pi_{\B}(x)(y \otimes u) = xy \otimes u
\]
for all $h \in G$, $y \in B_h$, and $u \in \Hh_B$. 
In particular,~the unit fibre corresponds to the homogeneous component $B \otimes_B \Hh_B$. 
Under the identification 
\[ 
    B \otimes_B \Hh_B \cong \Hh_B, \qquad x \otimes u \mapsto\pi_B(x)u, 
\] 
the action of $B$ on this component agrees with $\pi_B$. 
In this sense, $\pi_{\B}$ is a lift of $\pi_B$ to $C_r^*(\B)$.

\subsection{The fibrewise Dirac operator}

The smooth factor-system data allow the Dirac operator $D_B$ to be lifted fibrewise to the homogeneous components of $\Hh_{\B}$. 
Indeed, fix $g \in G$.
The unitary module map $\Phi_g$ from~\eqref{eq:fibre-decomposition} induces a unitary
\begin{equation}\label{eq:W_g}
    W_g: p(g) (\Hh_B \otimes H_g) \rightarrow B_g \otimes_B \Hh_B.
\end{equation}
Since $p(g) \in B_0 \otimes \cL(H_g)$, the commutator $[D_B \otimes \id,p(g)]$ is bounded. 
Thus, by~\cite[Lem.~2.1]{SchWa20}, the compression $p(g) (D_B \otimes \id) p(g)$ is self-adjoint with compact resolvent on $p(g) (\Hh_B \otimes H_g)$, with domain $p(g) (\dom(D_B) \otimes H_g)$.

Transporting this compression through $W_g$ yields the operator
\[
    D_g := W_g p(g) (D_B \otimes \id) p(g) W_g^*
\]
on $B_g \otimes_B \Hh_B$, with domain
\[
    \dom(D_g) := W_g p(g) (\dom(D_B) \otimes H_g).
\]
Since unitary conjugation preserves self-adjointness and compact resolvent, $D_g$ is self-adjoint and has compact resolvent. 

For $g=e$, the normalization $H_e = \C$ and $s(e) = 1_B$ gives $p(e) = 1_B$ and $W_e = \id_{\Hh_B}$, and hence $D_e = D_B$. 
The construction is summarized in the following lemma:

\begin{lemma}\label{lem:fibrewise-operators} 
For each $g \in G$, the operator $D_g$ is self-adjoint and has compact resolvent. Moreover, $D_e = D_B$. 
\end{lemma}

Passing to the Hilbert space direct sum in~\eqref{eq:Hilbert-space-decomposition}, it follows from \cite[Lem.~5.3.7]{Ped89} that there is a unique self-adjoint operator $D_{\mathrm{fib}}$ on $\Hh_{\B}$ such that $D_{\mathrm{fib}} \vert_{\dom(D_g)} = D_g$ for all $g \in G$. 
Equivalently,
\[
    D_{\mathrm{fib}} := \bigoplus_{g \in G} D_g
\]
with domain
\[
    \dom(D_{\mathrm{fib}})
    :=
    \left\{
        (\xi_g)_{g \in G}:
        \xi_g \in \dom(D_g)
        \text{ for all } g \in G,
        \sum_{g \in G} \| D_g \xi_g \|^2 < \infty
    \right\}.
\]
The operator $D_{\mathrm{fib}}$ will be called the \emph{fibrewise Dirac operator}. 
Its basic properties are recorded below:

\begin{corollary}\label{cor:horizontal-operator}
The following assertions hold:
\begin{enumerate}
    \item $D_{\mathrm{fib}}$ is self-adjoint.
    \item For each $g\in G$, the homogeneous subspace $B_g \otimes_B \Hh_B$ reduces $D_{\mathrm{fib}}$, and the restriction of $D_{\mathrm{fib}}$ to this subspace is $D_g$.
    \item Under the canonical unitary identification $B \otimes_B \Hh_B \cong \Hh_B$, $x \otimes u \mapsto \pi_B(x)u$, the restriction of $D_{\mathrm{fib}}$ to the $e$-homogeneous component agrees with $D_B$.
\end{enumerate}
\end{corollary}


\subsection{Spectral regularity}

The next step is to establish that $D_{\mathrm{fib}}$ has bounded commutators with $C_c(\B)_0$.
For all $g \in G$, $x \in B_g$, and $h \in G$, define
\[
    T_x^h: B_h \otimes_B \Hh_B \rightarrow B_{gh} \otimes_B \Hh_B,
    \qquad
    T_x^h(y \otimes u) := xy \otimes u.
\]
With respect to the decomposition~\eqref{eq:Hilbert-space-decomposition}, one then has
\[
    \pi_{\B}(x) = \bigoplus_{h \in G} T_x^h.
\]

\begin{definition}\label{def:spectral-regularity}
The smooth factor-system data are called \emph{spectrally regular} relative to $\mathbb D_B$ if, for all $g \in G$ and $x \in B_g^0$, the following conditions hold:
\begin{enumerate}
    \item For each $h \in G$, one has
    \[
        T_x^h \dom(D_h) \subseteq \dom(D_{gh}),
    \]
    and the operator
    \[
        \nabla_{x,h} := D_{gh} T_x^h-T_x^h D_h
    \]
    on $\dom(D_h)$ extends to a bounded operator from
    $B_h \otimes_B \Hh_B$ to $B_{gh} \otimes_B \Hh_B$.
    \item The family $(\nabla_{x,h})_{h\in G}$ is uniformly bounded, i.e.,
    \[
        \sup_{h\in G} \| \nabla_{x,h} \| < \infty.
    \]
\end{enumerate}
\end{definition}

Spectral regularity provides the uniform control needed for bounded commutators with $D_{\mathrm{fib}}$:

\begin{lemma}\label{lem:fibrewise-commutator}
Suppose that the factor-system data are spectrally regular relative to $\mathbb D_B$. Then
\[
    \pi_{\B}(C_c(\B)_0)\dom(D_{\mathrm{fib}})
    \subseteq
    \dom(D_{\mathrm{fib}})
    \qquad
    \text{and}
    \qquad
    [D_{\mathrm{fib}},\pi_{\B}(\xi)] \in \cL(\Hh_{\B})
\]
for all $\xi \in C_c(\B)_0$.
\end{lemma}
\begin{proof} 
It suffices to consider $x \in B_g^0$ for some $g \in G$. 
Let $\xi=(\xi_h)_{h \in G} \in \dom(D_{\mathrm{fib}})$. 
By spectral regularity, $T_x^h \xi_h \in \dom(D_{gh})$, $D_{gh} T_x^h \xi_h = T_x^h D_h \xi_h + \nabla_{x,h} \xi_h$ for all $h \in G$, and 
\[
    \sup_{h \in G} \| \nabla_{x,h} \| < \infty.
\]
Since $\| T_x^h \| \leq \| x \|$ for all $h \in G$ and $\xi \in \dom(D_{\mathrm{fib}})$, it follows that
\[
    \sum_{h \in G} \| D_{gh} T_x^h \xi_h \|^2 < \infty.
\]
Thus $\pi_{\B}(x)\xi \in \dom(D_{\mathrm{fib}})$, and consequently $\pi_{\B}(x)\dom(D_{\mathrm{fib}}) \subseteq \dom(D_{\mathrm{fib}})$.
The bijectivity of $h\mapsto gh$ implies that the operators $\nabla_{x,h}$ have pairwise orthogonal source and target summands.
Therefore, on $\dom(D_{\mathrm{fib}})$,
\[
    [D_{\mathrm{fib}},\pi_{\B}(x)] = \bigoplus_{h\in G} \nabla_{x,h},
    \qquad
    \text{and}
    \qquad
    \| [D_{\mathrm{fib}},\pi_{\B}(x)] \| = \sup_{h\in G} \|\nabla_{x,h} \|
    < \infty.
\]
Hence the commutator extends to a bounded operator on $\Hh_{\B}$.
\end{proof}

\subsection{A factor-system criterion}

Spectral regularity can be expressed in terms of the factor-system data.
Through the representation $\pi_B$, elements of $B \otimes \cL(K,L)$ are regarded as operators from $\Hh_B \otimes K$ to $\Hh_B \otimes L$ for finite-dimensional Hilbert spaces $K$ and~$L$.

For all $h \in G$ and $x \in B_0$, put
\[
    \partial\gamma_h(x)
    :=
    [D_B \otimes \id,\gamma_h(x)].
\]
For all $g,h \in G$, put
\[
    \partial\omega(g,h)
    :=
    (D_B \otimes \id) \omega(g,h) - \omega(g,h) (D_B \otimes \id).
\]
By smoothness, these operators extend to bounded operators for each fixed choice of indices.

For each $h\in G$, put
\[
    \psi_h
    :=
    W_h p(h): \Hh_B\otimes H_h \rightarrow B_h \otimes_B \Hh_B.
\]
By construction, $\psi_h$ is surjective.
Under the identification in~\eqref{eq:W_g}, one has
\[
    \psi_h(u \otimes \vartheta)
    =
    a_h(1_B \otimes \vartheta) \otimes u
\]
for all $u \in \Hh_B$ and $\vartheta \in H_h$.
Consequently,~\eqref{eq:factor-multiplication} gives
\begin{equation}\label{eq:coordinate-representation}
    \pi_{\B}(a_g(x \otimes \eta))
    \psi_h(u \otimes \vartheta)
    =
    \psi_{gh}(\omega(g,h) \gamma_h(x)_{13} (u \otimes \eta \otimes \vartheta))
\end{equation}
for all $g,h \in G$, $x \in B$, $\eta \in H_g$, $u \in \Hh_B$, and $\vartheta \in H_h$.

\begin{theorem}[Factor-system criterion]\label{thm:factor-criterion}
Suppose that the smooth factor-system data satisfy
\[
    \sup_{h \in G} \| \partial\gamma_h(x) \| < \infty
\]
for all $x \in B_0$, and
\[
    \sup_{h \in G} \| \partial\omega(g,h) \| < \infty
\]
for all $g \in G$.
Then the factor-system data are spectrally regular relative to $\mathbb D_B$. 
In~particular,
\[
    [D_{\mathrm{fib}},\pi_{\B}(\xi)] \in \cL(\Hh_{\B})
\]
for all $\xi \in C_c(\B)_0$.
\end{theorem}
\begin{proof}
By linearity, it suffices to consider $z = a_g(x \otimes \eta)$ for some $g \in G$, $x \in B_0$, and~$\eta \in H_g$. 

Fix $h \in G$ and define
\[
    V_\eta^h:
    \Hh_B \otimes H_h \rightarrow \Hh_B \otimes H_g \otimes H_h,
    \qquad
    V_\eta^h(u \otimes \vartheta) := u \otimes \eta \otimes \vartheta.
\]
Put $\widetilde T_z^h := \omega(g,h) \gamma_h(x)_{13} V_\eta^h$.
Equation~\eqref{eq:coordinate-representation} then gives $T_z^h \psi_h = \psi_{gh} \widetilde T_z^h$.

Smoothness implies that $\gamma_h(x)_{13}$ preserves the corresponding domain of the amplification of $D_B$, while $\omega(g,h)$ maps this domain into $\dom(D_B) \otimes H_{gh}$. 
The map $V_\eta^h$ intertwines the respective amplifications.
Hence
\[
    \widetilde T_z^h
    (\dom(D_B) \otimes H_h)
    \subseteq
    \dom(D_B) \otimes H_{gh}.
\]
Moreover, the range relation $p(gh) \omega(g,h) = \omega(g,h)$ gives $p(gh) \widetilde T_z^h = \widetilde T_z^h$.
Consequently, 
\[ 
    \widetilde T_z^h (p(h)(\dom(D_B) \otimes H_h)) 
    \subseteq 
    p(gh)(\dom(D_B) \otimes H_{gh}). 
\]

For each $\zeta \in p(h)(\dom(D_B) \otimes H_h)$, one therefore has
\[
    T_z^h \psi_h(\zeta) = \psi_{gh}(\widetilde T_z^h \zeta)
    \in
    \dom(D_{gh}),
\]
and thus $T_z^h \dom(D_h) \subseteq \dom(D_{gh})$, as required.

On the same domain,
\[
    \nabla_{z,h} \psi_h(\zeta)
    =
    \psi_{gh}(\partial\omega(g,h) \gamma_h(x)_{13} V_\eta^h \zeta) 
    + 
    \psi_{gh}(\omega(g,h) \partial\gamma_h(x)_{13} V_\eta^h \zeta).
\]

Since $\psi_h$ restricts to a unitary from $p(h)(\Hh_B \otimes H_h)$ onto $B_h \otimes_B \Hh_B$, $\gamma_h$ is contractive, $\| \omega(g,h) \| \leq 1$, and $\|V_\eta^h\|=\|\eta\|$, it follows that
\[
    \| \nabla_{z,h} \|
    \leq
    \| \eta \| (\| \partial\omega(g,h) \| \| x \| + \| \partial\gamma_h(x) \|).
\]
Thus $\nabla_{z,h}$ extends to a bounded operator from $B_h \otimes_B \Hh_B$ to $B_{gh} \otimes_B \Hh_B$.

By the assumptions of the theorem,
\[
    \sup_{h \in G} \| \nabla_{z,h} \|
    \leq
    \| \eta \|
    \left(\| x \| \sup_{h \in G} \| \partial\omega(g,h) \| + \sup_{h \in G} \| \partial\gamma_h(x) \|\right)
    < \infty.
\]

This establishes spectral regularity of the factor-system data relative to $\mathbb D_B$.
The final assertion is a consequence of Lemma~\ref{lem:fibrewise-commutator}.
\end{proof}

\section{The lifted spectral triple}
\label{sec:construction}

Retain the saturated Fell bundle $\B=(B_g)_{g \in G}$ with unital unit fibre $B=B_e$, equipped with the spectral triple $\mathbb D_B=(B_0,\pi_B,\Hh_B,D_B)$, and the smooth factor-system data relative to $B_0$ fixed in Section~\ref{sec:lifts}.
These data determine the Hilbert space $\Hh_{\B}$, the smooth \Star-algebra $C_c(\B)_0$, and the fibrewise operator $D_{\mathrm{fib}}$.
Under the spectral-regularity assumptions, $D_{\mathrm{fib}}$ has bounded commutators with $C_c(\B)_0$.

\subsection{Matrix-valued length functions and the length operator}

To control the group direction, the vertical geometry associated with matrix-valued weights studied by Masters~\cite{MastersThesis} is adapted to the present setting. 
Let $V$ be a finite-dimensional Hilbert space.
Following~\cite{MastersThesis}, a function $\ell: G \rightarrow \cL(V)_{\mathrm{sa}}$ is called~\emph{translation bounded} if
\[
    \sup_{h \in G} \| \ell(gh)-\ell(h) \| < \infty
\]
for all $g \in G$, and \emph{proper} if, for each $R \geq 0$, the set
\[
    \{g \in G: \min\operatorname{Spec} |\ell(g)| \leq R\}
\]
is finite.
A proper translation-bounded function $\ell: G \rightarrow \cL(V)_{\mathrm{sa}}$ with $\ell(e) = 0$ will be called a \emph{matrix-valued length function}.

Let $\ell: G \rightarrow \cL(V)_{\mathrm{sa}}$ be a matrix-valued length function.
Using the decomposition
\[
    \Hh_{\B} \otimes V
    \cong
    \bigoplus_{g \in G} (B_g \otimes_B \Hh_B) \otimes V,
\]
define the \emph{length operator}
\begin{equation}\label{eq:Mell}
    M_\ell(\xi_g)_{g \in G} := ((\id \otimes \ell(g)) \xi_g)_{g \in G}
\end{equation}
with domain
\[
    \dom(M_\ell)
    :=
    \left\{
        (\xi_g)_{g \in G}:
        \sum_{g \in G} \| (\id \otimes \ell(g))\xi_g \|^2
        <\infty
    \right\}.
\]
The operator $M_\ell$ is self-adjoint, being the maximal diagonal multiplication operator associated with the self-adjoint operators $\ell(g)$, $g \in G$.
For each $g \in G$, the homogeneous component $(B_g \otimes_B \Hh_B) \otimes V$ reduces $M_\ell$, and its restriction to this component is $\id_{B_g \otimes_B \Hh_B} \otimes \ell(g)$.

The following is the discrete saturated version of the bounded-commutator estimate underlying the vertical unbounded Kasparov module constructed in~\cite{MastersThesis}. It is included for completeness and to fix the notation used below.

\begin{lemma}\label{lem:vertical-commutator}
For all $g \in G$ and $x \in B_g$, $(\pi_{\B}(x) \otimes \id_V) \dom(M_\ell) \subseteq \dom(M_\ell)$, and the commutator $[M_\ell,\pi_{\B}(x) \otimes \id_V]$ extends to a bounded operator on $\Hh_{\B} \otimes V$ satisfying
\[
    \| [M_\ell,\pi_{\B}(x) \otimes \id_V] \|
    \leq
    c_\ell(g) \| x \|,
\]
where
\[
    c_\ell(g) := \sup_{h \in G} \| \ell(gh)-\ell(h) \| < \infty.
\]
\end{lemma}
\begin{proof}
Fix $g \in G$ and $x \in B_g$, and let $\xi = (\xi_h)_{h \in G} \in \dom(M_\ell)$.
For each $h \in G$, one has $\| T_x^h \| \leq \| x \|$ and $\| \ell(gh) - \ell(h) \| \leq c_\ell(g)$.
Hence
\[
    \sum_{h\in G} \| (\id \otimes \ell(gh)) (T_x^h\otimes\id_V)\xi_h \|^2
    \leq
    2 \| x \|^2 \sum_{h\in G} ( \| (\id \otimes \ell(h)) \xi_h \|^2 + c_\ell(g)^2 \| \xi_h \|^2)
    < \infty.
\]
Thus $(\pi_{\B}(x) \otimes \id_V) \xi \in \dom(M_\ell)$.

On $\dom(M_\ell)$,
\[
    [M_\ell,\pi_{\B}(x) \otimes \id_V] = \bigoplus_{h \in G} T_x^h \otimes (\ell(gh)-\ell(h)).
\]
Therefore, $\| [M_\ell,\pi_{\B}(x) \otimes \id_V] \| \leq c_\ell(g)\|x\|$, which proves the assertion.
\end{proof}

\subsection{The lifted Dirac operator}

The fibrewise operator $D_{\mathrm{fib}}$ and the length operator $M_\ell$
are combined to form the Dirac operator of the lifted spectral triple.
The two contributions are separated using the first two Pauli matrices:

Let
\[
    \sigma_1
    :=
    \begin{pmatrix}
        0&1\\
        1&0
    \end{pmatrix},
    \qquad
    \sigma_2
    :=
    \begin{pmatrix}
        0&-i\\
        i&0
    \end{pmatrix}
\]
be the first two Pauli matrices. 

On $\Hh_{\B} \otimes V \otimes \C^2$, define
\[
    D_{\B,\ell} := D_{\mathrm{fib}} \otimes \id_V \otimes \sigma_1 + M_\ell \otimes \sigma_2
\]
with domain
\[
    \dom(D_{\B,\ell}) := (\dom(D_{\mathrm{fib}} \otimes \id_V) \cap \dom(M_\ell)) \otimes \C^2.
\]
The intersection $\dom(D_{\mathrm{fib}} \otimes \id_V) \cap \dom(M_\ell)$ is dense in $\Hh_{\B} \otimes V$, because it contains
\[
    \bigoplus_{g \in G}^{\mathrm{alg}} (\dom(D_g) \otimes V).
\]

\begin{lemma}\label{lem:compact-resolvent}
The operator $D_{\B,\ell}$ is self-adjoint and has compact resolvent.
\end{lemma}
\begin{proof}
With respect to the decomposition~\eqref{eq:Hilbert-space-decomposition}, the restrictions of $D_{\B,\ell}$ to the homogeneous components are the operators
\[
    D_{g,\ell} := D_g \otimes \id_V \otimes \sigma_1 + \id \otimes \ell(g) \otimes \sigma_2,
    \qquad
    g \in G.
\]
For each $g \in G$, the operator $D_{g,\ell}$ is self-adjoint on $\dom(D_g) \otimes V \otimes \C^2$, because the first summand is self-adjoint and the second is bounded and self-adjoint.

By~\cite[Lem.~5.3.7]{Ped89}, the operator $\bigoplus_{g\in G} D_{g,\ell}$ is self-adjoint on
\[
    \left\{
        (\zeta_g)_{g \in G}:
        \zeta_g \in \dom(D_g) \otimes V \otimes \C^2
        \text{ for all }g \in G,\ 
        \sum_{g \in G} \| D_{g,\ell} \zeta_g \|^2 
        < \infty
    \right\}.
\]
Moreover, the anticommutation relation $\sigma_1 \sigma_2 + \sigma_2 \sigma_1 = 0$ implies that, for all
$g \in G$ and $\zeta \in \dom(D_g) \otimes V \otimes \C^2$,
\[
    \| D_{g,\ell} \zeta \|^2
    =
    \| (D_g \otimes \id_V \otimes \sigma_1) \zeta \|^2 + \| (\id \otimes \ell(g) \otimes \sigma_2) \zeta \|^2.
\]
It follows that the natural domain of $\bigoplus_{g \in G} D_{g,\ell}$ is 
\[
   (\dom(D_{\mathrm{fib}} \otimes \id_V) \cap \dom(M_\ell)) \otimes \C^2 = \dom(D_{\B,\ell}).
\]
Thus $D_{\B,\ell} = \bigoplus_{g \in G} D_{g,\ell}$, and $D_{\B,\ell}$ is self-adjoint.

For each $g \in G$, the anticommutation relation $\sigma_1 \sigma_2 + \sigma_2 \sigma_1 = 0$ also gives
\[
    D_{g,\ell}^2 = (D_g^2 \otimes \id_V + \id \otimes \ell(g)^2) \otimes \id_{\C^2}
\]
on $\dom(D_g^2) \otimes V \otimes \C^2$.
Set $m_\ell(g) := \min\operatorname{Spec} |\ell(g)|$.
Since $\ell(g)^2 \geq m_\ell(g)^2 \id_V$,
\[
    \id + D_{g,\ell}^2 \geq (1 + m_\ell(g)^2) \id,
\]
and hence
\[
    \| (\id + D_{g,\ell}^2)^{-1} \| \leq \frac{1}{1 + m_\ell(g)^2}.
\]

Each operator $(\id + D_{g,\ell}^2)^{-1}$, $g \in G$, is compact, because $D_g$ has compact resolvent, $V$ is finite-dimensional, and $\id \otimes \ell(g)^2$ is bounded.
Moreover, by properness of $\ell$, $m_\ell(g) \to \infty$ at infinity on $G$. 
Therefore, 
\[
    \| (\id + D_{g,\ell}^2)^{-1} \| \to 0
\]
at infinity. 
The decomposition
\[
    (\id + D_{\B,\ell}^2)^{-1} = \bigoplus_{g \in G} (\id + D_{g,\ell}^2)^{-1},
\]
together with the norm decay above, shows that $(\id  +D_{\B,\ell}^2)^{-1}$ is compact.
Thus  $D_{\B,\ell}$ has compact resolvent.
\end{proof}

\subsection{The lifting theorem}

The preceding constructions can now be combined to obtain a spectral triple on the reduced cross-sectional C\Star-algebra:

\begin{theorem}[Fell-bundle lifting theorem]
\label{thm:main}
Let $G$ be a countable discrete group, let $V$ be a finite-dimensional Hilbert space, and let
$\ell: G \rightarrow \cL(V)_{\mathrm{sa}}$ be a matrix-valued length function.
Let $\B=(B_g)_{g \in G}$ be a saturated Fell bundle with unital unit fibre $B$, and let $\mathbb D_B=(B_0,\pi_B,\Hh_B,D_B)$ be a spectral triple on $B$.
Suppose that $\B$ is equipped with factor-system data that are smooth relative to $B_0$ and spectrally regular relative to $\mathbb D_B$.
Then
\begin{equation}\label{eq:lifted-triple}
    \mathbb D_{\B,\ell}
    :=
    \left(
        C_c(\B)_0,
        \pi_{\B} \otimes \id_V \otimes \id_{\C^2},
        \Hh_{\B} \otimes V \otimes \C^2,
        D_{\B,\ell}
    \right)
\end{equation}
is a spectral triple on $C_r^*(\B)$.
\end{theorem}
\begin{proof}
The algebra $C_c(\B)_0$ is a dense unital \Star-subalgebra of $C_r^*(\B)$, and the representation $\pi_{\B} \otimes \id_V \otimes \id_{\C^2}$ is faithful.
By Lemma~\ref{lem:compact-resolvent}, the operator $D_{\B,\ell}$ is self-adjoint and has compact resolvent.

It remains to verify the domain and commutator conditions.
Let $\xi \in C_c(\B)_0$.
By~Lemmas~\ref{lem:fibrewise-commutator} and~\ref{lem:vertical-commutator}, together with amplification by
$\id_V$, one has
\[
    (\pi_{\B}(\xi) \otimes \id_V \otimes \id_{\C^2})
    \dom(D_{\B,\ell})
    \subseteq
    \dom(D_{\B,\ell}).
\]
Moreover, on $\dom(D_{\B,\ell})$,
\[
    [D_{\B,\ell},\pi_{\B}(\xi) \otimes \id_V \otimes \id_{\C^2}]
    =
    [D_{\mathrm{fib}},\pi_{\B}(\xi)] \otimes \id_V \otimes \sigma_1
    +
    [M_\ell,\pi_{\B}(\xi) \otimes \id_V]
    \otimes \sigma_2.
\]
Both commutators on the right-hand side are bounded, and hence so is the commutator on the left-hand side. 
Since $\xi \in C_c(\B)_0$ was arbitrary, $\mathbb D_{\B,\ell}$ is a spectral triple on $C_r^*(\B)$.
\end{proof}

Combining Theorems~\ref{thm:factor-criterion} and~\ref{thm:main} yields the following consequence:

\begin{corollary}[Factor-system lifting criterion]
\label{cor:factor-system-main}
Let $G$ be a countable discrete group, let $V$ be a finite-dimensional Hilbert space, and let
$\ell: G \rightarrow \cL(V)_{\mathrm{sa}}$ be a matrix-valued length function.
Let $\B=(B_g)_{g \in G}$ be a saturated Fell bundle with unital unit fibre $B$, and let $\mathbb D_B=(B_0,\pi_B,\Hh_B,D_B)$ be a spectral triple on $B$.
Suppose that $\B$ is equipped with factor-system data that are smooth relative to $B_0$ and satisfy
\[
    \sup_{h \in G} \| \partial\gamma_h(x) \| < \infty
\]
for all $x \in B_0$, and
\[
    \sup_{h \in G} \| \partial\omega(g,h) \| < \infty
\]
for all $g \in G$.
Then the lifted quadruple $\mathbb D_{\B,\ell}$ in~\eqref{eq:lifted-triple} is a spectral triple on $C_r^*(\B)$.
\end{corollary}

\subsection{Kasparov-product interpretation}

The lifted spectral triple also admits a Kasparov-product interpretation. 
We refer to~\cite[Secs.~17--18]{Bla19} for background on $KK$-theory and the Kasparov product.

\begin{proposition}
\label{prop:kasparov-product}
Assume the hypotheses of Theorem~\ref{thm:main}, and suppose that $B$ is separable.
Regard the spectral triple on the unit fibre as an odd unbounded Kasparov module representing $[D_B] \in KK^1(B,\C)$, and let $[M_\ell] \in KK^1(C_r^*(\B),B)$ be the class of the vertical unbounded Kasparov module associated with $\ell$, as in~\cite{MastersThesis}.

Equip $\Hh_{\B} \otimes V \otimes \C^2$ with the grading
\[
    \Gamma := -\id_{\Hh_{\B}} \otimes \id_V \otimes \sigma_3,
    \qquad
    \sigma_3=
    \begin{pmatrix}
        1 &0 \\
        0 & -1
    \end{pmatrix}.
\]
Then the lifted spectral triple, with this grading, represents
\[
    [D_{\B,\ell},\Gamma] = [M_\ell ]\otimes_B[D_B] \in KK^0(C_r^*(\B),\C).
\]
\end{proposition}
\begin{proof}
Set $A := C_r^*(\B)$ and $E := \ell^2(\B) \otimes V$.
Since $B$ is separable, $G$ is countable, and each $B_g$, $g \in G$, is finitely generated as a right $B$-module, the algebra $A$ is separable.
On the Hilbert $B$-module $E$, let $M_\ell$ denote the regular self-adjoint multiplication operator
\[
    (M_\ell \xi)_g :=(\id_{B_g} \otimes \ell(g)) \xi_g,
    \qquad
    g \in G,
\]
on its maximal domain.
The vertical unbounded Kasparov module is $(E,M_\ell)$.

Under the canonical identification $E \otimes_B \Hh_B \cong \Hh_{\B} \otimes V$, the induced operator agrees with the Hilbert-space length operator defined in~\eqref{eq:Mell}.
The notation $M_\ell$ is used for both operators.

On $\Hh_{\B}\otimes V$, write $H := D_{\mathrm{fib}} \otimes \id_V$ and $L := M_\ell$.

To use the odd--odd convention of~\cite[Sec.~I.4]{MastersThesis}, set
\[
    \widetilde D := L \otimes \sigma_1 + H \otimes \sigma_2,
    \qquad
    \widetilde \Gamma := \id \otimes \sigma_3.
\]
The unitary
\[
    U := \id_{\Hh_{\B} \otimes V} \otimes \frac{\sigma_1 + \sigma_2}{\sqrt{2}}
\]
satisfies $U D_{\B,\ell} U^* = \widetilde D$, $U \Gamma U^* = \widetilde \Gamma$ and commutes with the representation of $A$.
It therefore suffices to identify the class of $(\widetilde D,\widetilde \Gamma)$.

To this end, Kucerovsky's unbounded Kasparov-product criterion~\cite[Thm.~13]{Kuc97}, in the odd--odd formulation of~\cite[Thm.~7.2]{KaLesch13} (with the corresponding Pauli-matrix convention), is verified.

Let $\xi \in B_g^0$.
Define 
\[
    T_\xi: \Hh_B \to \Hh_{\B},
    \qquad
    T_\xi u := \xi \otimes u.
\]
Under the identification $B \otimes_B \Hh_B \cong \Hh_B$, this is the operator $T_\xi^e$ introduced before Definition~\ref{def:spectral-regularity}.
Since $D_e = D_B$, spectral regularity gives $T_\xi \dom(D_B) \subseteq \dom(D_{\mathrm{fib}})$ and a bounded extension $K_\xi := D_{\mathrm{fib}}T_\xi - T_\xi D_B$.
Self-adjointness implies the corresponding adjoint domain inclusion and the identity
\[
    D_B T_\xi^* - T_\xi^* D_{\mathrm{fib}} = -K_\xi^*
\]
on $\dom(D_{\mathrm{fib}})$.

Let $w \in V$. 
Define $T_{\xi,w}u := T_\xi u \otimes w$ and $K_{\xi,w}u := K_\xi u \otimes w$.
Then $LT_{\xi,w} = T_{\xi,\ell(g)w}$ and $HT_{\xi,w} - T_{\xi,w} D_B = K_{\xi,w}$.
Consequently,
\[
    \widetilde D(T_{\xi,w} \otimes \id_{\C^2}) - (T_{\xi,w} \otimes \id_{\C^2}) (D_B \otimes \sigma_2)
    =
    T_{\xi,\ell(g)w} \otimes \sigma_1 + K_{\xi,w} \otimes \sigma_2.
\]
The right-hand side is bounded, and so is the corresponding adjoint connection operator.
This establishes the connection condition on the dense subspace
\[
    \left(\bigoplus_{g \in G}^{\mathrm{alg}} B_g^0 \right) \otimes V \subseteq E.
\]

For the domain condition, set $S := L\otimes\sigma_1$.
By construction,
\[
    \dom(\widetilde D) = (\dom(H)\cap\dom(L)) \otimes \C^2 \subseteq \dom(S).
\]

Finally, $H$ and $L$ strongly commute: each homogeneous component reduces both operators, and their restrictions act on different tensor factors.
The Pauli anticommutation relation therefore gives
\[
    \langle S\zeta,\widetilde D\zeta\rangle + \langle\widetilde D\zeta,S\zeta\rangle = 2\|S\zeta\|^2 \geq 0
\]
for $\zeta \in \dom(\widetilde D)$.
The identity follows first on finitely supported vectors and then on the full domain by truncation.
Thus the positivity condition holds, and the Kasparov-product criterion proves the assertion.
\end{proof}

\subsection{Summability}

A summability criterion for the lifted spectral triple is established next.
Recall that a spectral triple $(C_0,\pi,\Hh,D)$ is called
\emph{$p$-summable} if $(\id_{\Hh} + D^2)^{-p/2}$ is trace class.
The summability of $\mathbb D_{\B,\ell}$ can be estimated in terms of
uniform summability bounds for the fibre operators $D_g$, $g \in G$, and the growth of the matrix-valued length function $\ell$.

\begin{proposition}\label{prop:summability}
Assume the hypotheses of Theorem~\ref{thm:main}.
Suppose that there exist $p,q>0$ and $C>0$ such that
\[
    \operatorname{Tr}((\id+D_g^2)^{-p/2}) \leq C
\]
for all $g \in G$, and
\[
    \sum_{g \in G} \operatorname{Tr}_V((\id_V+\ell(g)^2)^{-q/2}) < \infty.
\]
Then
\[
    (\id_{\Hh_{\B} \otimes V \otimes \C^2} + D_{\B,\ell}^2)^{-(p+q)/2}
\]
is trace class.
\end{proposition}
\begin{proof}
For $a,b \geq 0$ and $r,s > 0$, one has
\[
    (1 + a + b)^{-(r + s)} \leq (1 + a)^{-r}(1 + b)^{-s}.
\]
Since $D_g^2 \otimes \id_V$ and $\id \otimes \ell(g)^2$ are commuting positive operators, functional calculus gives
\[
    (\id + D_g^2 \otimes \id_V + \id \otimes \ell(g)^2)^{-(p+q)/2}
    \leq
    (\id + D_g^2)^{-p/2} \otimes (\id_V + \ell(g)^2)^{-q/2}.
\]
Therefore,
\[
\begin{aligned}
    \operatorname{Tr}((\id + D_{\B,\ell}^2)^{-(p+q)/2})
    &\leq
    2
    \sum_{g \in G} \operatorname{Tr}((\id + D_g^2)^{-p/2}) \operatorname{Tr}_V((\id_V + \ell(g)^2)^{-q/2})
    \\
    &\leq
    2C
    \sum_{g \in G}
    \operatorname{Tr}_V((\id_V+\ell(g)^2)^{-q/2})
    < \infty,
\end{aligned}
\]
which establishes the claim.
\end{proof}

\section{Cocycle twists}
\label{sec:cocycle-twisting}

Retain the saturated Fell bundle $\B=(B_g)_{g\in G}$ with unital unit fibre $B=B_e$, the spectral triple $\mathbb D_B=(B_0,\pi_B,\Hh_B,D_B)$, and the smooth factor-system data relative to $B_0$ fixed in Section~\ref{sec:lifts}.
The stability of the lifted spectral triple under central cocycle twists of $\B$ is considered in this section.

\subsection{The cocycle-twisted Fell bundle}

The cocycle twisting of a saturated Fell bundle is recalled first,
together with the corresponding smooth algebra and the operators entering
the lifted spectral triple.

Let $\psi: G \rightarrow \Pic(B)$ be the Picard homomorphism associated with $\B$.
The induced action of $G$ on $U(Z(B))$ allows one to consider the twisted group cohomology associated with $\psi$, and in particular normalized $2$-cocycles
\[
    \nu \in Z^2(G,U(Z(B)))_\psi.
\]
For such a cocycle $\nu$, the corresponding cocycle twist, denoted by $\B_\nu$, is again a saturated Fell bundle with unit fibre $B$ and the same underlying fibres $B_g$, $g \in G$; see~\cite[Sec.~4]{MaWa25}.
Its multiplication and involution are given by
\[
    x \cdot_\nu y = \nu(g,h) xy,
    \qquad
    x^{*_\nu} = \nu(g^{-1},g)^*x^*
\]
for all $g,h \in G$, $x \in B_g$, and $y \in B_h$.
Here the products and involution on the right-hand side are those of $\B$.

To preserve the same smooth homogeneous subspaces under the cocycle twist, assume in addition that $\nu(g,h) \in B_0$ for all $g,h \in G$.
Since left multiplication by elements of $B_0$ preserves each smooth fibre $B_k^0$, $k \in G$, one has $\nu(g,h) B_k^0 \subseteq B_k^0$ for all $g,h,k \in G$.
Together with the previously established relations
\[
    B_g^0 \cdot B_h^0 \subseteq B_{gh}^0
    \qquad
    \text{and}
    \qquad
    (B_g^0)^* \subseteq B_{g^{-1}}^0
\]
for all $g,h \in G$, this gives
\[
    B_g^0 \cdot_\nu B_h^0 \subseteq B_{gh}^0
    \qquad
    \text{and}
    \qquad
    (B_g^0)^{*_\nu} \subseteq B_{g^{-1}}^0
\]
for all $g,h \in G$.
Thus
\[
    C_c(\B_\nu)_0 := \bigoplus_{g\in G}^{\mathrm{alg}}B_g^0
\]
is a dense unital \Star-subalgebra of $C_r^*(\B_\nu)$.

The cocycle twist does not change the underlying Hilbert $B$-bimodules, so that $\Hh_{\B_\nu} = \Hh_{\B}$.
Under this identification, the fibre operators $D_g$, $g \in G$, and hence $D_{\mathrm{fib}} = \bigoplus_{g \in G}D_g$, remain unchanged.
Their amplification $D_{\mathrm{fib}} \otimes \id_V$ is therefore unchanged as well.
The length operator $M_\ell$ on $\Hh_{\B}\otimes V$ depends only on the grading and the matrix-valued length function $\ell$, and is likewise unaffected.
Consequently, set
\[
    D_{\B_\nu,\ell} := D_{\B,\ell}
\]
on $\Hh_{\B}\otimes V\otimes\C^2$.
In particular, $D_{\B_\nu,\ell}$ is self-adjoint and has compact resolvent.

For $g,h \in G$ and $x \in B_g$, the operator $\pi_{\B_\nu}(x)$ maps the $h$-homogeneous component into the $gh$-homogeneous component. 
Its restriction to this component is given by 
\begin{equation}\label{eq:twisted-T} 
    T_{x,\nu}^h = L_{\nu(g,h)}^{gh} T_x^h, 
\end{equation} 
where, for $z \in B_0$ and $k \in G$, 
\[ 
    L_z^k: B_k \otimes_B \Hh_B \rightarrow B_k \otimes_B \Hh_B 
\] 
denotes the operator induced by left multiplication by $z$. 
Indeed, for $y \in B_h$ and $u \in \Hh_B$, 
\[ 
    T_{x,\nu}^h (y \otimes u) = (x \cdot_\nu y) \otimes u = \nu(g,h) xy \otimes u. 
\] 
If $z \in U(Z(B)) \cap B_0$, then $L_z^k$ is unitary for each $k \in G$. 
The argument establishing faithfulness of $\pi_{\B}$ in Section~\ref{sec:lifts} applies verbatim to $\pi_{\B_\nu}$.

\subsection{Stability of the lifted spectral triple}

Conditions ensuring that spectral regularity, and hence the lifted spectral
triple, is preserved under cocycle twisting are determined next.

Suppose from now on that the factor-system data are spectrally regular relative to $\mathbb D_B$.
Since $B_e^0 = B_0$ and $T_z^k = L_z^k$, spectral regularity applied to $z \in B_e^0$ implies
\[
    L_z^k \dom(D_k) \subseteq \dom(D_k)
    \qquad
    \text{and}
    \qquad
    [D_k,L_z^k]
    \in
    \cL(B_k \otimes_B \Hh_B)
\]
for all $z \in B_0$ and $k \in G$.

For $x\in B_g^0$, spectral regularity of the original factor-system data, together with the~domain invariance of $L_{\nu(g,h)}^{gh}$, implies 
\begin{equation}\label{eq:domain_invariance}
    T_{x,\nu}^h \dom(D_h) \subseteq \dom(D_{gh}) 
\end{equation}
for all $g,h \in G$.
Thus only uniform boundedness of the corresponding commutators remains to be controlled.

\begin{theorem}[Stability under cocycle twisting]
\label{thm:cocycle-twist}
Assume the hypotheses of Theorem~\ref{thm:main}, and let $\nu \in Z^2(G,U(Z(B)))_\psi$
be a normalized $2$-cocycle with respect to the action of $G$ on
$U(Z(B))$ induced by the Picard homomorphism $\psi: G \rightarrow \Pic(B)$. 
Suppose that $\nu(g,h) \in B_0$ for all $g,h \in G$ and that
\begin{equation}\label{eq:cocycle-uniform-bound}
    \sup_{h\in G} \| [D_{gh},L_{\nu(g,h)}^{gh}] T_x^h \| < \infty
\end{equation}
for all $g\in G$ and $x\in B_g^0$.
Then
\[
    \left(
        C_c(\B_\nu)_0,
        \pi_{\B_\nu} \otimes \id_V \otimes \id_{\C^2},
        \Hh_{\B} \otimes V \otimes \C^2,
        D_{\B_\nu,\ell}
    \right)
\]
is a spectral triple on $C_r^*(\B_\nu)$.
\end{theorem}
\begin{proof}
Fix $g \in G$ and $x \in B_g^0$. 
For each $h \in G$, define on
$\dom(D_h)$
\[
    \nabla_{x,h}^\nu := D_{gh} T_{x,\nu}^h - T_{x,\nu}^h D_h.
\]
Using~\eqref{eq:twisted-T}, one finds that
\[
    \nabla_{x,h}^\nu = [D_{gh},L_{\nu(g,h)}^{gh}] T_x^h + L_{\nu(g,h)}^{gh} (D_{gh} T_x^h - T_x^h D_h).
\]
Both terms on the right-hand side extend to bounded operators, and hence so does $\nabla_{x,h}^\nu$. 
Since $L_{\nu(g,h)}^{gh}$ is unitary, its bounded extension satisfies
\[ 
    \| \nabla_{x,h}^\nu \| 
    \leq 
    \| [D_{gh},L_{\nu(g,h)}^{gh}] T_x^h \| + \| \nabla_{x,h} \|. 
\] 
It follows from this estimate, Condition~\eqref{eq:cocycle-uniform-bound}, and spectral regularity that
\[ 
    \sup_{h \in G} \| \nabla_{x,h}^\nu \| < \infty. 
\]
Combined with the domain invariance in~\eqref{eq:domain_invariance}, this shows that the twisted homogeneous operators satisfy the spectral-regularity conditions.
Consequently,
\[
    \pi_{\B_\nu}(C_c(\B_\nu)_0)\dom(D_{\mathrm{fib}})
    \subseteq
    \dom(D_{\mathrm{fib}})
    \qquad
    \text{and}
    \qquad
    [D_{\mathrm{fib}},\pi_{\B_\nu}(\xi)]
    \in
    \cL(\Hh_{\B})
\]
for all $\xi \in C_c(\B_\nu)_0$.

For all $g,h \in G$ and $x \in B_g$, $\| T_{x,\nu}^h \| = \| T_x^h \| \leq \| x \|$, by~\eqref{eq:twisted-T} and the unitarity of $L_{\nu(g,h)}^{gh}$.
For such $g$ and $x$, the reasoning used in the proof of Lemma~\ref{lem:vertical-commutator} therefore yields
\[
    (\pi_{\B_\nu}(x) \otimes \id_V) \dom(M_\ell)
    \subseteq
    \dom(M_\ell)
    \qquad
    \text{and}
    \qquad
    \| [M_\ell,\pi_{\B_\nu}(x) \otimes \id_V] \|
    \leq
    c_\ell(g) \| x \|,
\]
where
\[
    c_\ell(g) := \sup_{h \in G} \| \ell(gh) - \ell(h) \|.
\]
Linearity extends the corresponding domain invariance and boundedness to all $\xi \in C_c(\B_\nu)_0$.

For all $\xi \in C_c(\B_\nu)_0$, the preceding domain-invariance and commutator-boundedness statements imply, respectively,
\[
    (\pi_{\B_\nu}(\xi) \otimes \id_V \otimes \id_{\C^2})
    \dom(D_{\B_\nu,\ell})
    \subseteq
    \dom(D_{\B_\nu,\ell}),
\]
and that, on $\dom(D_{\B_\nu,\ell})$,
\[
    [D_{\B_\nu,\ell},\pi_{\B_\nu}(\xi) \otimes \id_V \otimes \id_{\C^2}]
    =
    [D_{\mathrm{fib}},\pi_{\B_\nu}(\xi)] \otimes \id_V \otimes \sigma_1 + [M_\ell,\pi_{\B_\nu}(\xi) \otimes \id_V] \otimes \sigma_2,
\]
extends to a bounded operator on $\Hh_{\B} \otimes V \otimes \C^2$.

The fact that $C_c(\B_\nu)_0$ is a dense unital \Star-subalgebra of $C_r^*(\B_\nu)$, the faithfulness of $\pi_{\B_\nu}$, and the self-adjointness and compact resolvent of $D_{\B_\nu,\ell}$ were established above. 
Hence
\[
    \left(
        C_c(\B_\nu)_0,
        \pi_{\B_\nu} \otimes \id_V \otimes \id_{\C^2},
        \Hh_{\B} \otimes V \otimes \C^2,
        D_{\B_\nu,\ell}
    \right)
\]
is a spectral triple on $C_r^*(\B_\nu)$.
\end{proof}

A convenient sufficient condition for~\eqref{eq:cocycle-uniform-bound} is obtained by estimating the cocycle commutators independently of $x$. Indeed, for $g,h\in G$ and $x\in B_g^0$,
\[
    \|[ D_{gh},L_{\nu(g,h)}^{gh}] T_x^h \|
    \leq
    \| [D_{gh},L_{\nu(g,h)}^{gh}] \| \| T_x^h \|
    \leq
    \| [D_{gh},L_{\nu(g,h)}^{gh}] \| \| x \|.
\]
Hence~\eqref{eq:cocycle-uniform-bound} follows whenever
\[
    \sup_{h\in G} \| [D_{gh},L_{\nu(g,h)}^{gh}] \| < \infty
\]
for all $g \in G$. 
This gives the following consequence:

\begin{corollary}\label{cor:cocycle-uniform}
Assume the hypotheses of Theorem~\ref{thm:main}, and let $\nu \in Z^2(G,U(Z(B)))_\psi$
be a normalized $2$-cocycle with respect to the action of $G$ on
$U(Z(B))$ induced by the Picard homomorphism $\psi: G \rightarrow \Pic(B)$.
Suppose that $\nu(g,h) \in B_0$ for all $g,h \in G$ and that
\[
    \sup_{h\in G} \|[ D_{gh},L_{\nu(g,h)}^{gh}] \| <\infty
\]
for all $g \in G$. 
Then
\[
    \left(
        C_c(\B_\nu)_0,
        \pi_{\B_\nu} \otimes \id_V \otimes \id_{\C^2},
        \Hh_{\B} \otimes V \otimes \C^2,
        D_{\B_\nu,\ell}
    \right)
\]
is a spectral triple on $C_r^*(\B_\nu)$.

In particular, the same conclusion holds if $[D_{gh},L_{\nu(g,h)}^{gh}] = 0$ for all $g,h \in G$.
\end{corollary}

\section{Examples}
\label{sec:examples}

The construction is now illustrated by a number of examples emphasizing
different features of the general framework.
For trivial Fell bundles, the lifted spectral triple reduces to the usual product construction, while twisted crossed products recover the standard spectral triples associated with matrix-valued length functions on the acting group.
Fell line bundles provide the corresponding twisted group C\Star-algebras as a particularly simple special case.

The finite-group and free-group examples exhibit nonfree homogeneous fibres.
Examples arising from free compact Abelian actions provide further comparisons: the quantum $4$-torus illustrates Clifford-valued vertical geometry, while the quantum Hopf fibration illustrates nontrivial projective spectral subspaces, subject to the stated analytic hypotheses.

Finally, the gauge grading of the Cuntz algebras illustrates the interaction between the growth of the fibre multiplicities and the summability properties of the lifted operator.

\subsection{The trivial Fell bundle}

Let $B$ be a unital C\Star-algebra equipped with a spectral triple $\mathbb D_B=(B_0,\pi_B,\Hh_B,D_B)$, let $G$ be a countable discrete group, let $V$ be a finite-dimensional Hilbert space, and let $\ell: G \rightarrow \cL(V)_{\mathrm{sa}}$ be a matrix-valued length function.

Consider the trivial Fell bundle $\B=(B_g)_{g\in G}$ over $G$,
defined by $B_g=B$ for all $g \in G$, with multiplication and involution inherited from $B$.
Its reduced cross-sectional algebra is canonically isomorphic to
\[
    C_r^*(\B) \cong B \otimes_{\min} C_r^*(G).
\]
Moreover,
\[
    \Hh_{\B} \cong \Hh_B\otimes\ell^2(G).
\]

Under these identifications, the representation $\pi_{\B}$ is identified with $\pi_B \otimes \lambda$, where $\lambda$ denotes the left regular representation of $G$. 
More precisely, for $x \in B_g = B$, $\pi_{\B}(x) = \pi_B(x)\otimes\lambda_g$.

For each $g \in G$, choose $H_g = \C$ and  $s(g) = 1_B$.
The corresponding factor-system data are $p(g)=1_B$, $\gamma_g = \id_B$, and $\omega(g,h) = 1_B$ for all $g,h \in G$, and are smooth relative to $B_0$.

Since $a_g = \id_B$ for all $g \in G$, $B_g^0 = B_0$ for all $g \in G$, and therefore
\[
    C_c(\B)_0 \cong B_0 \otimes C_c(G).
\]

For all $g \in G$, identify $B_g \otimes_B \Hh_B \cong \Hh_B$, $y \otimes u \mapsto \pi_B(y)u$.
With this convention, $D_g = D_B$ for all $g \in G$. 
Moreover, for all $g,h \in G$ and $x \in B_g^0=B_0$, the operator $T_x^h$ is identified
with $\pi_B(x)$, and hence
\[
    \nabla_{x,h} = [D_B,\pi_B(x)].
\]
The domain-invariance and boundedness conditions follow from the spectral
triple $\mathbb D_B$, and the right-hand side is independent of $h$.
Thus the factor-system data are spectrally regular relative to
$\mathbb D_B$.

Since $D_g = D_B$ for all $g\in G$, the fibrewise operator becomes
\[
    D_{\mathrm{fib}} = D_B \otimes \id_{\ell^2(G)}.
\]
On $\ell^2(G)\otimes V$, let $M_\ell$ denote the multiplication operator
defined by
\[
    M_\ell(\delta_g \otimes v) = \delta_g \otimes \ell(g)v
\]
for all $g \in G$ and $v \in V$.
Under the identification $\Hh_{\B} \otimes V \cong \Hh_B \otimes \ell^2(G) \otimes V$, the length operator is identified with $\id_{\Hh_B} \otimes M_\ell$.
Therefore,
\[
    D_{\B,\ell}
    =
    D_B \otimes \id_{\ell^2(G)} \otimes \id_V \otimes \sigma_1 + \id_{\Hh_B} \otimes M_\ell \otimes \sigma_2.
\]
Theorem~\ref{thm:main} gives the spectral triple
\[
    \left(
        B_0 \otimes C_c(G),
        \pi_B \otimes \lambda \otimes \id_V \otimes \id_{\C^2},
        \Hh_B \otimes \ell^2(G) \otimes V \otimes \C^2,
        D_{\B,\ell}
    \right)
\]
on $B \otimes_{\min} C_r^*(G)$.
Thus, for the trivial Fell bundle, the construction reduces to the product of the spectral triple on $B$ with the spectral data determined by the matrix-valued length function on $G$.

\subsection{Twisted crossed products}

Let $B$ be a unital C\Star-algebra equipped with a spectral triple $\mathbb D_B=(B_0,\pi_B,\Hh_B,D_B)$, let $G$ be a countable discrete group, let $V$ be a finite-dimensional Hilbert space, and let $\ell: G \rightarrow \cL(V)_{\mathrm{sa}}$ be a matrix-valued length function.

A normalized twisted action of $G$ on $B$ is a pair $(\rho,\sigma)$ consisting of maps
\[
    \rho: G \rightarrow \operatorname{Aut}(B),
    \qquad
    \sigma: G \times G \rightarrow U(B),
\]
where $U(B)$ denotes the unitary group of $B$, such that 
\[
    \rho_g \circ \rho_h
    =
    \operatorname{Ad}(\sigma(g,h)) \circ \rho_{gh},
\]
and
\[
    \rho_g(\sigma(h,k)) \sigma(g,hk)
    =
    \sigma(g,h) \sigma(gh,k)
\]
for all $g,h,k \in G$, normalized by 
\[
    \rho_e = \id_B,
    \qquad
    \sigma(e,g) = \sigma(g,e) = 1_B
\]
for all $g \in G$.

For each $g \in G$, let $u_g$ be a formal symbol. 
The twisted action determines a saturated Fell bundle $\B=(B_g)_{g \in G}$ with $B_g = Bu_g$, whose multiplication and involution are given by
\[
    (xu_g) (yu_h) = x \rho_g(y) \sigma(g,h) u_{gh}
\]
and
\[
    (xu_g)^* = \sigma(g^{-1},g)^* \rho_{g^{-1}}(x^*)u_{g^{-1}}
\]
for all $g,h \in G$ and $x,y \in B$.
Its reduced cross-sectional algebra is canonically isomorphic to the reduced twisted crossed product,
\[
    C_r^*(\B) \cong B \rtimes_{\rho,\sigma,r} G.
\]

The relation with the factor-system description from Section~\ref{sec:isometric-factor-systems} is particularly simple.
For each $g \in G$, choose $H_g = \C$ and $s(g) = u_{g^{-1}} \in B_{g^{-1}}$.
The twisted-action relations imply $s(g)^* s(g) = 1_B$, so that $p(g)=1_B$. 
The corresponding factor-system data are
\[
    \gamma_g(x) = \rho_{g^{-1}}(x)
\]
for all $g \in G$ and $x \in B$, and
\[
    \omega(g,h) = \sigma(h^{-1},g^{-1})^*
\]
for all $g,h \in G$.

Suppose now that the twisted action is smooth relative to $B_0$, in the sense that
\[
    \rho_g(B_0) = B_0
    \qquad
    \text{and}
    \qquad
    \sigma(g,h) \in U(B_0)
\]
for all $g,h \in G$.
The formulas above imply that the associated factor-system data are smooth relative to $B_0$. 

The coordinate map $a_g$ is given by
\[
    a_g(x) = s(g)^*x = \sigma(g,g^{-1})^* \rho_g(x )u_g
\]
for all $g \in G$ and $x \in B$. 
Hence $B_g^0 = B_0u_g$ for all $g \in G$, and consequently
\[
    C_c(\B)_0 \cong C_c(G,B_0).
\]

For each $k \in G$, the unitary $W_k$ from~\eqref{eq:W_g} identifies $\Hh_B$ with $B_k \otimes_B \Hh_B$. 
More explicitly,
\[
    W_k v
    =
    \sigma(k,k^{-1})^* u_k \otimes v
\]
for all $v \in \Hh_B$, while
\[
    W_k^*((yu_k) \otimes v)
    =
    \pi_B(\rho_{k^{-1}}(y) \sigma(k^{-1},k))v
\]
for all $y \in B$ and $v \in \Hh_B$.
For each $k \in G$, the corresponding fibre operator satisfies $D_k = W_k D_B W_k^*$.
Taking the direct sum of the unitaries $W_k$, $k \in G$, defines a unitary
\[
    W: \Hh_B \otimes \ell^2(G) \rightarrow \Hh_{\B},
    \qquad
    W(v \otimes \delta_k) := W_k v,
\]
and, because $D_{\mathrm{fib}} = \bigoplus_{k \in G} D_k$,
\[
    W^* D_{\mathrm{fib}} W = D_B \otimes \id_{\ell^2(G)}.
\]

For $g,h \in G$ and $x \in B$, the homogeneous operator $T_{xu_g}^h$ satisfies
\[
    W_{gh}^* T_{xu_g}^h W_h
    =
    \pi_B(\rho_{h^{-1}g^{-1}}(x) \sigma(h^{-1}g^{-1},g)),
\]
where the twisted-action identities have been used to obtain the right-hand side. 
Consequently,
\[
    W^*\pi_{\B}(xu_g) W(v \otimes \delta_h)
    =
    \pi_B(\rho_{h^{-1}g^{-1}}(x) \sigma(h^{-1}g^{-1},g)) v \otimes \delta_{gh}
\]
for all $v \in \Hh_B$.

Let $g \in G$ and $x \in B_0$, so that $xu_g \in B_g^0$.
For each $ h\in G$, the preceding identities give
\[
    W_{gh}^* \nabla_{xu_g,h} W_h
    =
    [D_B,\pi_B(\rho_{h^{-1}g^{-1}}(x) \sigma(h^{-1}g^{-1},g))].
\]
Since the twisted action is smooth relative to $B_0$,
\[
    \rho_{h^{-1}g^{-1}}(x)
    \sigma(h^{-1}g^{-1},g)
    \in B_0.
\]
Hence the domain-invariance condition and boundedness of $\nabla_{xu_g,h}$ for each $h\in G$ follow from the defining properties of the spectral triple $\mathbb D_B$.

For fixed $g\in G$, the map $h \mapsto h^{-1}g^{-1}$ is a bijection of $G$. 
Therefore,
\[
    \sup_{h \in G} \| \nabla_{xu_g,h} \|
    =
    \sup_{r \in G} \| [D_B,\pi_B(\rho_r(x)\sigma(r,g))] \|.
\]
Consequently, the factor-system data are spectrally regular relative to $\mathbb D_B$ provided that
\[
    \sup_{r \in G} \| [D_B,\pi_B(\rho_r(x)\sigma(r,g))] \| < \infty
\]
for all $g \in G$ and $x \in B_0$.
This is precisely the uniform boundedness condition in~\cite[Setup~3.1]{AnGuIsRu22}.

The length operator is carried by $W \otimes \id_V$ to $\id_{\Hh_B} \otimes M_\ell$, where $M_\ell$ is the multiplication operator on $\ell^2(G) \otimes V$ defined by $M_\ell(\delta_g \otimes v) = \delta_g \otimes \ell(g)v$ for all $g \in G$ and $v \in V$.
Set
\[
    \widetilde D_{\B,\ell} := (W^* \otimes \id_V \otimes \id_{\C^2}) D_{\B,\ell} (W \otimes \id_V \otimes \id_{\C^2}).
\]
Then
\[
    \widetilde D_{\B,\ell} = D_B \otimes \id_{\ell^2(G)} \otimes \id_V \otimes \sigma_1 + \id_{\Hh_B} \otimes M_\ell \otimes \sigma_2.
\]

Under the uniform boundedness condition above, Theorem~\ref{thm:main} yields the spectral triple
\[
    \left(
        C_c(G,B_0),
        W^*\pi_{\B}(\,\cdot\,)W \otimes \id_V \otimes \id_{\C^2},
        \Hh_B \otimes \ell^2(G) \otimes V \otimes \C^2,
        \widetilde D_{\B,\ell}
    \right)
\]
on the reduced twisted crossed product $B \rtimes_{\rho,\sigma,r} G$.
This agrees with the corresponding twisted crossed-product construction
in~\cite[Thm.~3.7]{AnGuIsRu22}.

\subsubsection{Fell line bundles}

As a special case, let $B = \C$ and let $\sigma \in Z^2(G,\mathbb T)$ be a normalized~$2$-cocycle.
The associated Fell bundle is a Fell line bundle, and
\[
    C_r^*(\B) \cong C_r^*(G,\sigma).
\]
Taking the canonical spectral triple of a point,
\[
    \mathbb D_{\C}=(\C,\id_{\C},\C,0),
\]
on the unit fibre, the fibre operators vanish, $D_g = 0$ for all $g\in G$, and spectral regularity is automatic.
The construction above therefore reduces to
the spectral triple
\[
    \left(
        C_c(G,\sigma),
        \lambda_\sigma \otimes\id_V \otimes \id_{\C^2},
        \ell^2(G) \otimes V \otimes \C^2,
        M_\ell \otimes \sigma_2
    \right)
\]
on $C_r^*(G,\sigma)$.

For $G = \Z^n$, Clifford-valued length functions provide the standard Dirac-type geometry in the group direction; see~\cite[Ex.~3.3]{AnGuIsRu22}.

\subsection{Finite Fell bundles}

Suppose that $G$ is finite. 
In this case, smoothness of the factor-system data already implies the uniform estimates in Corollary~\ref{cor:factor-system-main}. 
Indeed, for fixed $x \in B_0$ and $g \in G$, the families
\[
    \bigl(\partial\gamma_h(x)\bigr)_{h \in G}
    \qquad
    \text{and}
    \qquad
    \bigl(\partial\omega(g,h)\bigr)_{h \in G}
\]
contain only finitely many bounded operators. 
Hence every smooth factor system over a finite group is spectrally regular, and Theorem~\ref{thm:main} yields a spectral triple on $C_r^*(\B)$.

A natural class of examples is provided by free actions of finite Abelian groups. 
Let 
\[
    \alpha:K \rightarrow \operatorname{Aut}(A),
    \qquad
    k \mapsto \alpha_k,
\]
be an action of a finite Abelian group $K$ on a unital C\Star-algebra $A$, and let
\[
    B := A^K = \{x \in A: (\forall k \in K) \, \alpha_k(x) = x\}
\]
be the fixed-point algebra.
For each $\chi \in \widehat K$, consider the spectral subspace
\[
    A_\chi := \{x \in A: (\forall k \in K) \, \alpha_k(x) = \chi(k)x\}.
\]
Recall that, in this setting, freeness of the action is equivalent to
saturation of the spectral decomposition; equivalently, $\overline{A_\chi^* A_\chi} = B$ for all $\chi \in \widehat K$.
Hence the spectral subspaces~$(A_\chi)_{\chi \in \widehat K}$ form a saturated Fell bundle, with unit fibre $B$, and its reduced cross-sectional algebra is canonically isomorphic to $A$.

Let $\mathbb D_B$ be a spectral triple on $B$. 
If the associated factor-system data are smooth relative to $B_0$, then spectral regularity is automatic, because $\widehat K$ is finite. 
Therefore, Theorem~\ref{thm:main} yields a spectral triple on $A$.

Notably, although the action is free, the individual right $B$-modules $A_\chi$ need not be~free.
In general, the spectral subspaces are finitely generated projective $B$-modules, and the construction applies equally to nontrivial projective modules rather than only to free modules of rank one.

A finite non-Abelian example that does not arise from a twisted action is given next.

Let $K_1 := \C$, $K_2 := \C$, and let $K_3 := \C^2$, and set
\[
    F := \bigoplus_{i=1}^3 \cL(K_i) \cong \C \oplus \C \oplus M_2(\C).
\]
Note that $F\cong C^*(S_3)$. 
However, this identification will not be used below; the spaces $K_i$, $i=1,2,3$, are introduced to produce homogeneous fibres containing rectangular matrix blocks.

The symmetric group $S_3$ acts on $\{1,2,3\}$ in the usual way. 
For $g \in S_3$, define
\[
    E_g := \bigoplus_{i=1}^3 \cL(K_i,K_{g(i)}).
\]
In particular, $E_e = F$.
For $x=(x_i) \in E_g$ and $y=(y_i) \in E_h$, define multiplication and involution by
\[
    (xy)_i := x_{h(i)}y_i,
    \qquad
    (x^*)_i := x_{g^{-1}(i)}^*.
\]
These operations turn $\mathcal E := (E_g)_{g\in S_3}$ into a Fell bundle over $S_3$.

The bundle is saturated. 
Indeed, for $g,h \in S_3$ and $i \in \{1,2,3\}$, the $i$-th component of $E_gE_h$ is the linear span of operators of the form
\[
    ab,
    \qquad
    a \in \cL(K_{h(i)},K_{gh(i)}),
    \qquad
    b \in \cL(K_i,K_{h(i)}).
\]
Since $K_{h(i)}\neq0$, finite-dimensional matrix factorization gives
\[
    \operatorname{span}
    (\cL(K_{h(i)},K_{gh(i)}) \cL(K_i,K_{h(i)})) = \cL(K_i,K_{gh(i)}).
\]
Consequently, $[E_gE_h] = E_{gh}$ for all $g,h \in S_3$.

This Fell bundle is not of crossed-product type. 
To see this, consider the transposition $\tau=(13)$.
Then
\[
    E_\tau = \cL(\C,\C^2) \oplus \cL(\C,\C) \oplus \cL(\C^2,\C)
    \cong
    M_{2,1}(\C)\oplus\C\oplus M_{1,2}(\C).
\]
If $u=(u_1,u_2,u_3) \in E_\tau$ satisfied $u^*u = 1_F$, then necessarily $u_3^* u_3 = 1_{M_2(\C)}$.
This is~impossible, because $u_3 \in M_{1,2}(\C)$ and therefore $\operatorname{rank}(u_3^* u_3) \leq 1$.
Thus $E_\tau$ contains no unitary, and $\mathcal E$ is not a crossed-product Fell bundle.

For $\tau = (13)$, choose $H_\tau := \C^2$, with orthonormal basis $e_1,e_2$, and define $x_1,x_2 \in E_\tau$ by
\[
    x_1
    :=
    \left(
        \begin{pmatrix}
        1\\0
        \end{pmatrix},
        1,
        \begin{pmatrix}
        1&0
        \end{pmatrix}
    \right),
    \qquad
    x_2
    :=
    \left(
        \begin{pmatrix}
        0\\0
        \end{pmatrix},
        0,
        \begin{pmatrix}
        0&1
        \end{pmatrix}
    \right).
\]
Then
\[
    x_1^*x_1+x_2^*x_2=1_F,
\]
and hence
\[
    s(\tau) := x_1 \otimes e_1 + x_2 \otimes e_2 \in E_{\tau^{-1}} \otimes H_\tau
\]
is an isometry.

The associated projection $p(\tau) := s(\tau) s(\tau)^*$ is given, with respect to the decomposition
$F = \C \oplus \C \oplus M_2(\C)$, by
\[
    p(\tau)
    =
    \left(
        \begin{pmatrix}
            1&0\\
            0&1
        \end{pmatrix},
        \begin{pmatrix}
            1&0\\
            0&0
        \end{pmatrix},
        \begin{pmatrix}
            E_{11}&0\\
            0&0
        \end{pmatrix}
    \right)
    \in F\otimes M_2(\C).
\]
In particular,
\[
    p(\tau)\neq1_{F\otimes M_2(\C)}.
\]
The nontrivial projection $p(\tau)$ therefore gives an explicit factor-system realization of the nonfree homogeneous fibre $E_\tau$.

To obtain a nontrivial spectral example, let $C := C(\mathbb T)$, and consider the standard spectral triple
\[
    \mathbb D_C
    =
    \left(
        C^\infty(\mathbb T),
        \pi_C,
        L^2(\mathbb T),
        -i\frac{d}{d\theta}
    \right).
\]
Define
\[
    B_g := C(\mathbb T) \otimes E_g,
    \qquad
    g \in S_3.
\]
Then $\B=(B_g)_{g \in S_3}$ is a saturated Fell bundle with unit fibre $B :=  C(\mathbb T) \otimes F$.
Choose $B_0 := C^\infty(\mathbb T) \otimes F$ as the smooth subalgebra of $B$.
Let
\[
    K:=K_1 \oplus K_2 \oplus K_3
\]
and let $\pi_F$ denote the canonical representation of $F$ on $K$.
Then
\[
    \mathbb D_B
    :=
    \left(
        B_0,
        \pi_C\otimes\pi_F,
        L^2(\mathbb T)\otimes K,
        -i\frac{d}{d\theta}\otimes\id_K
    \right)
\]
is a spectral triple on $B$.

Choose factor-system data for $\mathcal E$ and define, for each $g \in S_3$, $s_B(g) := 1_C \otimes s(g)$.
Then, for each $g\in S_3$, $p_B(g) = 1_C \otimes p(g)$, and, for all $f \in C^\infty(\mathbb T)$ and $x \in F$,
\[
    \gamma_g^B(f \otimes x) = f \otimes \gamma_g^F(x).
\]
Moreover, for all $g,h \in S_3$,
\[
    \omega_B(g,h) = 1_C \otimes \omega_F(g,h).
\]
Since the factor-system data are constant in the $\mathbb T$-direction
and take values in the finite-dimensional algebra $F$, they preserve $B_0$.
Hence the induced factor-system data for $\B$ are smooth relative to
$B_0$.

It follows that, for all $g,h \in S_3$,
\[
    \partial\omega_B(g,h) = (D_B \otimes \id) \omega_B(g,h) - \omega_B(g,h) (D_B \otimes \id) = 0,
\]
while, for all $g \in S_3$, $f \in C^\infty(\mathbb T)$, and $x \in F$,
\[
    \partial\gamma_g^B(f\otimes x) = -if'\otimes\gamma_g^F(x).
\]
Since $S_3$ is finite, spectral regularity follows from the observation at the beginning of this subsection.

Let $\ell: S_3 \rightarrow \mathbb R$ be the word length associated with the generators $(12)$ and $(23)$. 
It is proper and translation bounded, so Theorem~\ref{thm:main} yields a spectral triple on $C_r^*(\B)$.

This gives a saturated Fell bundle over the non-Abelian group $S_3$ which is not of twisted crossed-product type: the homogeneous fibre corresponding to $\tau$ is not free of rank one over the unit fibre.
The nontrivial projection $p_B(\tau)$ makes this projective-module structure explicit in the factor-system description.

\subsection{A free-group grading with nonfree fibres}
\label{sec:free-group-example}

We construct a saturated Fell bundle over a non-Abelian free group with infinitely many pairwise nonisomorphic homogeneous modules.
Its spectral regularity follows directly from the Leibniz rule for tensor-product connections.

Let
\[
    M := S^2 \times S^1,
    \qquad
    B := C(M),
    \qquad
    B_0 := C^\infty(M),
\]
and equip $M$ with its product Riemannian metric and spin structure.
Write $\mathcal S$ for the spinor bundle and let
\[
    \mathbb D_B := (B_0,\pi_B,L^2(M,\mathcal S),D_M)
\]
be the corresponding odd spin spectral triple.

Let $L_0 \to S^2$ be a Hermitian line bundle of degree one, and let $L \to M$ be its pullback.
For $k \in \Z$, write $L^k$ for its $k$th tensor power, using the dual bundle when $k<0$.
In particular, $L^0$ is the trivial line bundle.

Set $G := F_2 =\langle a,b\rangle$ and let $n:G\to\Z$ be the homomorphism determined by $n(a)=1$ and $n(b)=0$.

For each $g \in G$, define $B_g := \Gamma(M,L^{n(g)})$.
The canonical tensor-product identifications
\[
    L^{n(g)} \otimes L^{n(h)} \cong L^{n(gh)}
\]
and Hermitian duality define multiplication and involution.
With the supremum norms, these operations make
$\B := (B_g)_{g\in G}$ a saturated Fell bundle with unit fibre $B$.

This Fell bundle is the pullback along $n:F_2 \to \Z$ of the $\Z$-graded Fell bundle associated with the line bundle $L$, whose spectral geometry fits into the generalized
crossed-product framework of~\cite{GaGr13}.

The modules $B_{a^k}$, $k \in \Z$, are pairwise nonisomorphic, since their associated line bundles have first Chern classes $k\,c_1(L)$.
In particular, $B_a$ is not a free right $B$-module.
Therefore $\B$ does not arise from a twisted action of $G$ on $B$.

Choose a smooth Hermitian projective realization of $L$ and equip it with the associated Grassmann connection $\nabla$.
Taking tensor products and duals gives compatible projective realizations of all $L^k$, whose Grassmann connections
agree with the induced tensor-product and dual connections $\nabla^k$.
These realizations and the canonical tensor-product identifications determine factor-system data $(\gamma_g,\omega(g,h))_{g,h \in G}$ that are smooth relative to $B_0=C^\infty(M)$.
The smooth homogeneous subspaces are given by
\[
    B_g^0 := \Gamma^\infty(M,L^{n(g)}),
    \qquad g \in G.
\]
For each $g \in G$, the natural identification
\[
    B_g \otimes_B L^2(M,\mathcal S)
    \cong L^2(M,L^{n(g)} \otimes \mathcal S)
\]
identifies the compressed operator used in the fibrewise construction with the spin Dirac operator twisted by
$\nabla^{n(g)}$, denoted by $D_g$.

Let $g,h \in G$, let $\xi \in B_g^0$, and let $T_\xi^h$ be multiplication by $\xi$.
The tensor-product~connections are compatible with multiplication, so the Leibniz rule gives
\[
    (D_{gh} T_\xi^h - T_\xi^h D_h)u = c(\nabla^{n(g)} \xi)u
\]
for every smooth section
$u$ of $L^{n(h)}\otimes\mathcal S$.
Here $c$ denotes Clifford multiplication on the spinor factor,
combined with multiplication of the line-bundle factors.
Consequently,
\[
    \|D_{gh} T_\xi^h - T_\xi^h D_h\|
    \leq
    \|c (\nabla^{n(g)}\xi)\|_\infty,
\]
where the right-hand side is independent of $h$.
Since smooth sections form a core for each twisted Dirac
operator, the same identity implies
\[
    T_\xi^h\dom(D_h)\subseteq\dom(D_{gh}).
\]
Thus the chosen factor-system data are spectrally regular.

A signed length function is chosen next.
For a nonempty reduced word $w \in F_2$, let $|w|$ denote its word length and put
\[
    \varepsilon(w):=
    \begin{cases}
        1, & \text{if the last letter of $w$ is $a$ or $b$},\\
        -1, & \text{if the last letter of $w$ is $a^{-1}$ or $b^{-1}$}.
    \end{cases}
\]
Define $\ell$ by $\ell(e) := 0$ and $\ell(w) := \varepsilon(w)|w|$ for all $w \neq e$.
Then $|\ell(w)| = |w|$, so $\ell$ is proper.
Moreover,
\[
    |\ell(gw) - \ell(w)| \leq |g|
    \qquad
    g,w \in F_2.
\]
To see this, if a nonempty suffix of $w$ survives reduction of $gw$, then the last letter is unchanged and the estimate follows from the usual word-length inequality.
Otherwise, all of $w$ cancels and $|gw| + |w| = |g|$, which gives the same estimate.
Hence $\ell$ is translation bounded.

Theorem~\ref{thm:main}, with $V=\C$, therefore gives a spectral triple on $C_r^*(\B)$ with Dirac operator
\[
    D_{\B,\ell} = D_{\mathrm{fib}} \otimes \sigma_1 + M_\ell \otimes \sigma_2.
\]

\begin{remark}
This example also separates spectral regularity from the sufficient conditions of Theorem~\ref{thm:factor-criterion}.
Indeed, for each $g \in G$, $\gamma_g(1) = p(g)$, so that criterion would require
\[
    \sup_{g \in G} \|[D_M,p(g)]\| < \infty.
\]
For each $k \in \Z$, the projection $p(a^k)$ represents the line bundle $L^k$.
A uniform bound on these Dirac commutators would yield uniform bounds on the derivatives of the projections and hence on the curvatures of their Grassmann connections.
Writing
\[
    F_k = p(a^k)(dp(a^k))^2
\]
for the corresponding curvature, the rank-one property implies that $\operatorname{Tr}(F_k)$ is uniformly bounded,
independently of the ambient matrix dimensions.
Consequently, the Chern numbers on $S^2 \times \{1\}$ would be uniformly bounded as $k$ varies over $\Z$.
This contradicts
\[
    \left|
    \frac{1}{2\pi i}
    \int_{S^2 \times \{1\}} \operatorname{Tr}(F_k)
    \right|
    = |k|,
    \qquad k \in \Z.
\]
Thus spectral regularity holds, although the stronger factor-system criterion does not.
\end{remark}

\subsection{Free compact Abelian actions}

Let $K$ be a compact metrizable Abelian group acting freely on a unital C\Star-algebra $A$.
Then $\widehat K$ is countable. 
The spectral subspaces $(A_\chi)_{\chi \in \widehat K}$ form a saturated Fell bundle with unit fibre $B = A^K$, and its reduced cross-sectional algebra is canonically isomorphic to $A$.

Unlike in the finite case, smoothness alone no longer guarantees the uniform estimates required for spectral regularity. 
Hence Theorem~\ref{thm:main} applies whenever the associated factor-system data are smooth and spectrally regular relative to a spectral triple on $B$.

This recasts the construction of~\cite{SchWa20} in the framework of saturated Fell bundles. 
The fibre operators $D_\chi$, $\chi \in \widehat K$, and hence $ D_{\mathrm{fib}}$, describe the geometry inherited from the fixed-point algebra $B$, while the length operator describes the direction of the dual group $\widehat K$.

\subsubsection{Quantum 4-tori}

A concrete comparison is provided by the quantum $4$-torus considered in~\cite[Sec.~7]{SchWa20}. 
Let $\mathbb A_\theta^4$ be the quantum $4$-torus and consider the restricted gauge~action
\[
    \alpha: \mathbb T^2 \rightarrow \operatorname{Aut}(\mathbb A_\theta^4)
\]
which fixes the first two canonical generators and acts on the remaining two coordinate directions. 
Its fixed-point algebra is the quantum $2$-torus $\mathbb A_{\theta'}^2$, while the spectral subspaces are indexed by $\widehat{\mathbb T}^2 \cong \Z^2$.
In~\cite{SchWa20}, the canonical Dirac operator on $\mathbb A_{\theta'}^2$ is lifted horizontally and combined with the vertical operator
\[
    D_v = \partial_1^{\mathbb T^2}\otimes\sigma_1 + \partial_2^{\mathbb T^2}\otimes\sigma_2.
\]
Under the Fourier identification $L^2(\mathbb T^2) \cong \ell^2(\Z^2)$, the operator $D_v$ is identified with multiplication by the matrix-valued function $\ell: \Z^2 \rightarrow \cL(\C^2)_{\mathrm{sa}}$, $\ell(m,n) := m\sigma_1 + n\sigma_2$.
It is a matrix-valued length function, because it is translation bounded and
\[
    \ell(m,n)^2 = (m^2+n^2)\id_{\C^2},
\]
which proves properness.
The associated length operator $M_\ell$ is therefore unitarily equivalent to $D_v$, so the Fell-bundle construction retains both Clifford directions appearing in~\cite{SchWa20}.

\subsubsection{The quantum Hopf fibration}

A different feature of the Fell-bundle formulation is illustrated by the quantum Hopf fibration. Consider the standard circle action
\[
    \beta: \mathbb T \rightarrow \operatorname{Aut}(C(SU_q(2)))
\]
whose fixed-point algebra is the standard Podleś sphere $B = C(S_q^2)$.
The spectral subspaces $(A_n)_{n \in\Z}$ form a saturated Fell bundle with unit fibre $B$, and its reduced cross-sectional algebra is canonically isomorphic to $C(SU_q(2))$. 
The fibres $A_n$, $n \in \Z$, are the quantum line bundles associated with the quantum Hopf fibration. 
As right $B$-modules they are finitely generated projective and, in general, nonfree. 
This example therefore exhibits a feature that is absent from the twisted crossed-product situation, where the homogeneous fibres are free right $B$-modules of rank one.

The quantum Hopf fibration and its spectral geometry have also been investigated by Senior~\cite{Sen11}, who studied the extension of a modular spectral triple
on the Podle\'s sphere to $SU_q(2)$ via an unbounded Kasparov product. 
In that setting, the resulting operator fails to have bounded commutators with the algebra of $SU_q(2)$.

Given a spectral triple $\mathbb D_B$ on $C(S_q^2)$ for which the associated factor-system data are smooth and spectrally regular, Theorem~\ref{thm:main} yields a spectral triple on $C(SU_q(2))$.
Its vertical part is determined by a matrix-valued length function on $\Z$.
Senior's construction highlights the analytic difficulties that may arise when extending spectral triples from the Podle\'s sphere to $SU_q(2)$.
In the present setting, the smoothness and spectral regularity assumptions remain to be verified for specific choices of Dirac operator on the Podle\'s sphere. 
We therefore make no identification with a standard Dirac operator on $SU_q(2)$.
Nevertheless, the example illustrates how the Fell-bundle framework accommodates nontrivial projective spectral subspaces.

\subsection{The gauge grading of the Cuntz algebras}

Let $N \geq 2$, and let $\mathcal O_N$ be the Cuntz algebra generated by
isometries $S_1,\ldots,S_N$. 
The gauge action
\[
    \alpha: \mathbb T \rightarrow \operatorname{Aut}(\mathcal O_N),
    \qquad
    \alpha_z(S_j) = z S_j,
\]
gives rise to the Fell bundle $\B=(A_n)_{n \in \Z}$, where
\[
    A_n
    =
    \overline{\operatorname{span}}
    \{ S_\mu S_\nu^*: |\mu|-|\nu| = n\}.
\]
Its unit fibre is the fixed-point algebra
\[
    B := A_0 \cong \varinjlim M_{N^k}(\C).
\]

The Fell bundle $\B$ is saturated. 
Indeed, for each $n \geq 1$ and every word $\mu$ of length $n$, the Cuntz relations give $S_\mu^* S_\mu = 1_B$, and therefore $1_B \in A_{-n} A_n$.
Moreover,
\[
    \sum_{|\mu|=n}S_\mu S_\mu^*=1_B,
\]
so that $1_B \in A_n A_{-n}$.
Hence $\overline{A_{-n} A_n} = \overline{A_n A_{-n}} = B$ for every $n \geq 1$, and the characterization of saturation from Section~\ref{sec:isometric-factor-systems} shows that $\B$ is saturated.
This also follows from the graph description of $\mathcal O_N$: it is the graph C\Star-algebra associated with the graph consisting of one vertex and $N$ loops, whose gauge action is free.

The factor-system realizations of the positive and negative fibres have rather different forms. 
For $n \geq 1$, choose $H_n := \C^{N^n}$, with orthonormal basis $(e_\mu)_{|\mu|=n}$ indexed by words of length $n$, and
\[
    s(n) := \sum_{|\mu|=n} S_\mu^* \otimes e_\mu \in A_{-n} \otimes H_n.
\]
Then
\[
    s(n)^ *s(n) = \sum_{|\mu|=n}S_\mu S_\mu^* = 1_B,
\]
so that $s(n)$ is an isometry. 
Moreover,
\[
    p(n) := s(n) s(n)^* = 1_B \otimes \id,
\]
because $S_\mu^* S_\nu = \delta_{\mu,\nu} 1_B$ for words $\mu,\nu$ of the same length. 
Consequently,
\[
    A_n
    \cong
    B\otimes H_n
    \cong
    B^{N^n}
\]
as right Hilbert $B$-modules.

For the components of negative degree, choose, for each $n \geq 1$, $H_{-n} = \C$ and $s(-n) := S_1^n \in A_n$.
Since $S_1^{*n} S_1^n = 1_B$, $s(-n)$ is an isometry, while
\[
    p(-n) := s(-n) s(-n)^* = S_1^nS_1^{*n}.
\]
Thus $A_{-n} \cong p(-n)B$ as right Hilbert $B$-modules. 
In particular, the two directions of the grading exhibit markedly different projective-module realizations:
\[
    A_n \cong B^{N^n},
    \qquad
    A_{-n} \cong S_1^nS_1^{*n}B,
    \qquad
    n \geq 1.
\]
The positive-degree fibres therefore have exponentially increasing free
rank, whereas the negative-degree fibres are represented by nontrivial
projections in $B$.

Let $\mathbb D_B=(B_0,\pi_B,\Hh_B,D_B)$ be a spectral triple on $B$, and suppose that the factor-system data determined by the above isometries are smooth relative to $B_0$.
For $n \geq 1$, the projection $p(n)$ is the identity, and the fibre operator $D_n$ constructed in Section~\ref{sec:lifts} is unitarily equivalent to $D_B \otimes \id$.
Consequently, if the base triple is $p$-summable, then, for each $n \geq 1$,
\[
    \operatorname{Tr}
    ((\id + D_n^2)^{-p/2}) = N^n \operatorname{Tr}((\id + D_B^2)^{-p/2}).
\]

Hence the uniform fibrewise trace estimate in Proposition~\ref{prop:summability} fails exponentially along the positive-degree fibres. 
This failure reflects an actual obstruction to finite summability for $\mathcal O_N$, rather than merely a limitation of the criterion. 
Connes proved that the existence of a finitely summable unbounded Fredholm module over a unital C\Star-algebra implies the existence of a tracial state; see~\cite{Co89}. 
The Cuntz algebra $\mathcal O_N$ has no tracial state. 
Indeed, if $\tau$ were tracial, then
\[
    1 = \tau(1) = \sum_{j=1}^N \tau(S_j S_j^*) = \sum_{j=1}^N \tau(S_j^* S_j) = N,
\]
which is impossible for $N \geq 2$.

For the negative fibres, the corresponding operators are compressions of $D_B$ by the projections $p(-n) = S_1^n S_1^{*n}$.
Thus their summability behaviour is governed by the interaction between $D_B$ and this family of projections rather than by an increasing finite-dimensional multiplicity. 

The example therefore illustrates that the fibrewise geometry of a saturated Fell bundle can be strongly asymmetric even when the grading group is $\Z$.

Finally, saturation alone does not ensure that the above factor-system data are spectrally regular relative to $\mathbb D_B$.
To apply Theorem~\ref{thm:main}, spectral regularity must still be established for the chosen spectral triple on $B$.
The sufficient conditions of Theorem~\ref{thm:factor-criterion} provide one way of doing so.

\section*{Acknowledgements}

The author would like to thank Magnus Goffeng for pointing out the connection with Ada Masters' work and for helpful insights concerning the interpretation of the construction in terms of unbounded Kasparov products. 
The author is also grateful to Ada Masters for valuable discussions and for clarifying the relationship with her work on Fell bundles.

\bibliographystyle{abbrv}
\bibliography{short,SAFB}

\end{document}